\documentclass[11pt]{amsart}
\usepackage{amssymb, amsfonts, amsmath, amsthm} 

\newtheorem{theorem}{Theorem}[section]
\newtheorem{lemma}[theorem]{Lemma}
\newtheorem{proposition}[theorem]{Proposition}
\newtheorem{corollary}[theorem]{Corollary}

\theoremstyle{definition}
\newtheorem{definition}[theorem]{Definition}
\newtheorem{notation}[theorem]{Notation}
\newtheorem{remark}[theorem]{Remark}
\newtheorem{notation-and-remark}[theorem]{Notation and Remark}
\newtheorem{definition-and-remark}[theorem]{Definition and Remark}
\newtheorem{remark-and-notation}[theorem]{Remark and Notation}

\newtheorem{ad-hoc-item}[theorem]{  }

\numberwithin{equation}{section}

\newcommand{\cA}{ {\mathcal A} }
\newcommand{\bC}{ {\mathbb C} }
\newcommand{\cD}{ {\mathcal D} }
\newcommand{\cF}{ {\mathcal F} }
\newcommand{\cM}{ {\mathcal M} }
\newcommand{\bN}{ {\mathbb N} }
\newcommand{\cP}{ {\mathcal P} }
\newcommand{\bR}{ {\mathbb R} }

\newcommand{\bt}{ {\bf t} }
\newcommand{\bu}{ {\bf u} }
\newcommand{\bs}{ {\bf v} }
\newcommand{\bw}{ {\bf w} }
\newcommand{\ugamma}{ \underline{\gamma} }
\newcommand{\ua}{ \underline{a} }
\newcommand{\uI}{ \underline{i} }
\newcommand{\uJ}{ \underline{j} }

\newcommand{\cyctok}{ \mbox{c}_{1 \to k} }
\newcommand{\cyckto}{ \mbox{c}_{k \to 1} }

\newcommand{\murd}{ \mu^{(r,d)} }
\newcommand{\nurd}{ \nu^{(r,d)} }
\newcommand{\perm}{ \mbox{p} }

\newcommand{\tr}{ \mathrm{tr} }

\newcommand{\Ker}{ \mbox{Ker} }
\newcommand{\ans}{ ( a_n )_{n=1}^{\infty} }
\newcommand{\bns}{ ( b_n )_{n=1}^{\infty} }

\newcommand{\cciup}{ \stackrel{\vee}{c} }
\newcommand{\abubble}{ \stackrel{\circ}{a} }

\title[CLT for star-generators of $\mathbf{S_{\infty}}$,
and law of a GUE matrix]
{A central limit theorem for star-generators of 
\boldmath{$S_{\infty}$}, which relates to the law of 
a GUE matrix}

\author[C. K\"ostler]{Claus K\"ostler}
\address{Claus K\"ostler: School of Mathematical Sciences,
University College Cork, Ireland.}
\email{claus@ucc.ie}
\author[A. Nica]{Alexandru Nica}
\thanks{AN: research supported by a Discovery Grant from 
NSERC, Canada.}
\address{Alexandru Nica: Department of Pure Mathematics, 
University of Waterloo, Ontario, Canada.}
\email{anica@uwaterloo.ca}

\begin{document}

\begin{abstract}
It is well-known that, on a purely algebraic level, a simplified 
version of the Central Limit Theorem (CLT) can be proved in 
the framework of a non-commutative probability space, under the 
hypotheses that the sequence of non-commutative random variables 
we consider is {\em exchangeable} and obeys a certain vanishing 
condition of some of its joint moments.  In this approach
(which covers versions for both the classical CLT and the CLT 
of free probability), the determination of the resulting limit 
law has to be addressed on a case-by-case basis.  

In this paper we discuss an instance of the above theorem that takes 
place in the framework of the group algebra $\bC [ S_{\infty} ]$ of
the infinite symmetric group: the exchangeable sequence is provided 
by the {\em star-generators} of $S_{\infty}$, 
and the expectation functional used on $\bC [ S_{\infty} ]$ depends 
in a natural way on a parameter $d \in \bN$.  We identify precisely 
the limit distribution $\mu_d$ for this special instance of CLT, via 
a connection that $\mu_d$ turns out to have with the average empirical 
eigenvalue distribution of a random $d \times d$ GUE matrix.  Moreover, 
we put into evidence a multi-variate version of this result which 
follows from the observation that, on the level of calculations 
with pair-partitions, the (non-centered) star-generators are 
related to a (centered) exchangeable sequence of GUE matrices 
with independent entries.
\end{abstract}

\maketitle
\section{Introduction}

\subsection{A central limit theorem for 
star-generators of $S_{\infty}$.}

$\ $

\noindent
It is well-known that, on a purely algebraic level, a simplified version 
of the Central Limit Theorem (CLT) can be proved via a direct calculation 
of moments: one expresses the relevant moments as sums with terms indexed 
by set-partitions, and one keeps a record of which terms in those 
sums can actually contribute in the limit.  This proof works without 
changes when one moves to the framework of a non-commutative probability 
space.  Moreover, the argument still works when 
the considered sequence of non-commutative random variables
is not required to 
satisfy some form of independence, but is only required to have a weaker 
property of {\em exchangeability}, together with an additional 
vanishing condition on some of the sequence's joint moments 
(those joint moments which, in a certain sense, contain a 
``singleton''); see \cite{BoSp1996}, Theorem 0 on page 138.
A price to pay for these weakened hypotheses is that the 
resulting limit law is no longer universal, and its determination 
has to be addressed on a case-by-case basis.  In particular, 
this approach covers simplified ``combinatorial'' versions of both 
the classical CLT, where the limit is the normal law, and of the 
CLT of free probability, where the limit is the semicircle law of 
Wigner.

In the present paper we examine an interesting instance of this 
limit theorem for exchangeable sequences, taking place in the 
framework of the infinite symmetric group 
\begin{equation}   \label{eqn:1a}
S_{\infty} := \{ \tau : \bN \to \bN \mid 
\mbox{ $\tau$ is bijective and }
\exists \, k_o \in \bN \mbox{ such that $\tau (k) = k$ for 
$k > k_o$} \}.
\end{equation}
On the group algebra $\bC [ S_{\infty} ]$ we consider a natural 
expectation functional, as follows: we fix a
$d \in \bN = \{ 1,2, \ldots \}$ and we let 
$\varphi_d : \bC [ S_{\infty} ] \to \bC$ be the linear functional 
determined by the requirement that 
\begin{equation}   \label{eqn:1b}
\varphi_d ( \tau ) = (1/d)^{|| \tau ||}, 
\ \ \tau \in S_{\infty},
\end{equation}
where the length $|| \tau ||$ of a permutation 
$\tau \in S_{\infty}$ is defined as the minimal 
number $m$ of transpositions $\rho_1, \ldots , \rho_m$
needed in order to achieve  a factorization 
$\tau = \rho_1 \cdots \rho_m$.  The restriction of $\varphi_d$ 
to $S_{\infty}$ is an example of ``block character'' in the 
sense of \cite{GnGoKe2013}, and is at the same time an example of 
extremal character of $S_{\infty}$, parametrized by a very simple 
double-sequence in the Thoma classification of such extremal 
characters (see e.g. Section 3 of the survey paper \cite{Ol2003},
or Section 4.1 of the monograph \cite{BoOl2016}).  In particular,
the linear functional $\varphi_d$ is positive; that is, one has
\[
\varphi_d ( a^{*}a ) \geq 0 \mbox{ for all }
a \in \bC [S_{\infty} ],
\]
where the $*$-operation on $\bC [ S_{\infty} ]$ is introduced 
via the requirement that $\tau^{*} = \tau^{-1}$ for all 
$\tau \in S_{\infty}$.  Hence $( \bC [ S_{\infty} ], \varphi_d )$
is an example of $*$-probability space (by which we mean a couple 
$( \cA , \varphi )$ where $\cA$ is a unital $*$-algebra over $\bC$
and $\varphi : \cA \to \bC$ is a positive linear functional such
that $\varphi (1) = 1$).

A natural example of exchangeable sequence in 
$( \, \bC [ S_{\infty} ] , \varphi_d )$ is provided by its
so-called ``star-generators'', i.e. by the sequence of 
transpositions 
\begin{equation}   \label{eqn:1c}
\gamma_1 = (1, 2), \gamma_2 = (1,3), \ldots ,
\gamma_n = (1, n+1), \ldots
\end{equation}
This sequence of generators of $S_{\infty}$ has received some 
attention in the combinatorics literature, see e.g. the exposition 
and references included in the introduction of \cite{Fe2012}.  
The CLT for exchangeable sequences applies to the centerings of the 
$\gamma_n$'s and produces a 
probability distribution $\mu_d$ on $\bR$, which is the limit in 
moments, for $n \to \infty$, of the centered and normalized sums
\begin{equation}   \label{eqn:1d}
s_n = \frac{1}{\sqrt{n}} 
\Bigl( \, ( \gamma_1 + \cdots + \gamma_n ) 
- \frac{n}{d} \, \Bigr) \in \bC [ S_{\infty} ].
\end{equation}
In this paper we prove that the probability distribution $\mu_d$ 
is intimately related to the empirical eigenvalue distribution of 
a Gaussian Hermitian (usually referred to as ``GUE") random matrix 
of size $d \times d$.  More precisely, one has the following theorem.

\begin{theorem}  \label{thm:11}
Let $\mu_d$ be the limit in moments of the sequence (\ref{eqn:1d}) 
in $( \, \bC [ S_{\infty} ], \varphi_d )$.  Then
\begin{equation}     \label{eqn:11a}
\mu_d * N(0, 1/d^2) = \nu_d
\ \ 
\mbox{ (convolution of probability measures),}
\end{equation}
where $N( 0, 1/d^2 )$ is the centered normal distribution of 
variance $1/d^2$ and $\nu_d$ is the average empirical eigenvalue
distribution of a $d \times d$ GUE matrix of variance 1.
\end{theorem}

The convolution formula from Theorem \ref{thm:11} can be 
combined with known facts about the average empirical eigenvalue 
distribution of a GUE matrix in order to obtain further 
information about $\mu_d$.  One gets in particular 
a precise description of Laplace transform.  Indeed, it is
known (see e.g. Section 3.3 of the monograph \cite{AnGuZe2009}, 
or Section 2 of the survey paper \cite{HaTo2003}) 
that the Laplace transform of $\nu_d$ is defined for 
all $z \in \bC$, and has the explicit formula
\begin{equation}   \label{eqn:57e}
\int_{\bR} e^{zt} \ d \, \nu_d (t) = Q_d (z)
\cdot e^{z^2/(2d)}, 
\ \ z \in \bC ,
\end{equation}
where $Q_d (z)$ is the polynomial defined by
\begin{equation}   \label{eqn:57f}
Q_d (z) = \sum_{j=0}^{d-1} \frac{1}{d^j \, (j+1)!}
\cdot \left(  \begin{array}{c} d-1 \\ j \end{array} \right)
\cdot z^{2j}.
\end{equation}
Upon substituting this into the formula for Laplace transforms
which follows directly from Equation (\ref{eqn:11a}), one arrives 
to the following corollary.

$\ $

\begin{corollary}    \label{cor:12}
The Laplace transform (or equivalently, the exponential 
moment-generating function) of $\mu_d$ is defined for all 
$z \in \bC$, and has the explicit formula
\begin{equation}   \label{eqn:12a}
\int_{\bR} e^{zt} \ d \, \mu_d (t) = 
Q_d (z) \cdot e^{(d-1) z^2/(2d^2)}, 
\ \ z \in \bC ,
\end{equation}
where the polynomial $Q_d$ is picked from (\ref{eqn:57f}).
\end{corollary}

\begin{remark}   \label{rem:13}

\vspace{6pt}

(1) From Corollary \ref{cor:12} and general considerations on 
the Laplace transform it follows that, for $d \geq 2$, the 
probability measure $\mu_d$ is absolutely continuous 
\footnote{ For $d=1$ it is immediately inferred, either 
from Equation (\ref{eqn:11a}) or from Equation (\ref{eqn:12a}),
that $\mu_1$ is the Dirac mass at $0$.} 
with respect to Lebesgue measure, with density of the form 
$P_d (t) e^{-t^2 d^2 / (2d-2)}$, where $P_d$ is an even 
polynomial of degree $2d-2$.  Moreover, 
the coefficients of $P_d$ can be obtained out of those of 
the polynomial $Q_d$ from Equation (\ref{eqn:57f}) by 
solving a triangular system of linear equations.  For instance,
for $d=2$ and $d=3$ one gets
\[
d \mu_2 (t) = \frac{1}{\sqrt{2 \pi}} 8t^2 \, e^{-2t^2} \, dt,
\ \ d \mu_3 (t) = \frac{3}{64 \sqrt{\pi}} 
(20 -108 t^2 +243 t^4) \, e^{-9t^2/4} \, dt.    
\]
\vspace{6pt}

(2) The sum $\gamma_1 + \cdots + \gamma_n$ 
appearing in (\ref{eqn:1d}) is unitarily conjugated to the sum 
of transpositions
$(1, n+1) + (2, n+1) + \cdots + (n,n+1) \in \bC [ S_{\infty} ]$,
which goes under the name of ``Jucys-Murphy element''.  Due to 
this fact, Theorem \ref{thm:11} can be re-stated as a result 
about the limit distribution for the sequence of centered and 
renormalized Jucys-Murphy elements.  

\vspace{6pt}

(3) It is natural to ask what happens when in Theorem \ref{thm:11} 
we replace the expectation functional $\varphi_d$ by a functional
$\varphi : \bC [ S_{\infty} ] \to \bC$ coming from a more general 
extremal character of $S_{\infty}$.  The centered star-generators
will continue to be exchangeable in the newly considered 
$*$-probability space 
$( \bC [ S_{\infty} ], \varphi )$, but will no longer satisfy
the vanishing-moment condition.  In order to understand what are 
the correct setting and limit theorem for this more general situation, 
it is likely that one needs some considerations of non-commutative 
dynamical systems, along the lines started in \cite{GoKo2010}.

\vspace{6pt}

(4) In the limit case ``$d = \infty$'', i.e. when in the 
considerations leading to Theorem \ref{thm:11} we replace 
$1/d$ by $0$, the limit law $\mu_d$ becomes the semicircle 
law, and we retrieve a result of Biane \cite{Bi1995}.  

The paper \cite{Bi1995} also puts into evidence a phenomenon 
of asymptotic free independence, going in a multi-variate 
framework.  In connection to that, we prove that
Theorem \ref{thm:11} has a multi-variate version, which is 
discussed in the next subsection.
\end{remark}

$\ $

\subsection{Multi-variate extension of Theorem \ref{thm:11}.}

$\ $

\noindent
In order to state (in Theorem \ref{thm:16} below) the 
multi-variate version, we first introduce some relevant notation 
referring to non-commutative distributions.

\begin{definition-and-remark}   \label{def:14}
Let $r$ be in $\bN$.

(1) Let $\bC \langle X_1, \ldots , X_r \rangle$ be the algebra of
polynomials in non-commuting indeterminates $X_1, \ldots, X_r$.
We turn $\bC \langle X_1, \ldots , X_r \rangle$ into a 
$*$-algebra by considering on it the $*$-operation uniquely 
determined by the requirement that $X_p^{*} = X_p$ for 
$1 \leq p \leq r$.  We denote
\begin{equation}   \label{eqn:14a}
\cD_r = \{ 
\mu : \bC \langle X_1, \ldots , X_r \rangle \to \bC ,
\mbox{ linear } \mid
\mu \mbox{ is positive and has } \mu (1) = 1 \}
\end{equation}
(where the fact that $\mu$ is positive means, by definition,
that $\mu ( P^{*}P ) \geq 0$ for all 
$P \in \bC \langle X_1, \ldots , X_r \rangle$).  For a functional
$\mu \in \cD_r$, the quantities of the form 
\begin{equation}   \label{eqn:14b}
\mu (X_{i(1)} \cdots X_{i(k)}), 
\mbox{ with $k \in \bN$ and 
       $( i(1), \ldots , i(k) ) \in \{ 1, \ldots , r \}^k$,}
\end{equation}
are called {\em moments} of $\mu$.  On $\cD_r$ we have a natural 
notion of {\em convergence in moments}, where the convergence 
of a sequence $( \mu_n )_{n=1}^{\infty}$ to a limit $\mu$ 
amounts to the fact that
\[
\lim_{n \to \infty}  \mu_n (P) = \mu (P), 
\ \ \forall \, P \in 
\bC \langle X_1, \ldots , X_r \rangle .
\]

(2) Let $\mu$ be a functional in the space $\cD_r$ 
considered above.  The 
{\em commuting exponential generating function} 
(or ``commuting e.g.f", for short) for the moments 
of $\mu$ is the power series
\begin{equation}  \label{eqn:14c}
f(z_1, \ldots , z_r) := 1 + \sum_{k=1}^{\infty}
\ \sum_{i(1), \ldots , i(k) =1}^r
\ \frac{ \mu (X_{i(1)} \cdots X_{i(k)}) }{k!}
\, z_{i(1)} \cdots z_{i(k)},
\end{equation}
where $z_1, \ldots , z_r$ are commuting indeterminates.

We note that in the case $r=1$, (\ref{eqn:14c}) gives precisely
the usual e.g.f for the moments of a distribution $\mu$.  When 
$r \geq 2$, the series $f(z_1, \ldots , z_r)$ from (\ref{eqn:14c}) 
still says something about the moments of $\mu$, but one can only 
retrieve certain symmetric sums of such moments.  (E.g. the 
coefficient of $z_1 z_2$ allows us to find the sum 
$\mu (X_1 X_2) + \mu (X_2 X_1)$, but not the individual values 
of $\mu (X_1 X_2)$ and $\mu (X_2 X_1)$.)

\vspace{6pt}

(3) Let $( \cA , \varphi )$ be a $*$-probability space
and let $a_1, \ldots , a_r$ be an $r$-tuple of selfadjoint 
elements of $\cA$.  The {\em joint distribution} of 
$a_1, \ldots , a_r$ in $( \cA , \varphi )$ is the functional 
$\mu \in \cD_r$ defined via the requirement that for every 
$k \in \bN$ and every tuple
$( i(1), \ldots , i(k) ) \in \{ 1, \ldots , r \}^k$ one has
\[
\mu (X_{i(1)} \cdots X_{i(k)} ) =
\varphi (a_{i(1)} \cdots a_{i(k)} ).
\]
\end{definition-and-remark}

\begin{remark}   \label{rem:15}
Some of the notions from Definition \ref{def:14} 
can still be considered in connection to any linear functional 
$\mu : \bC \langle X_1, \ldots , X_r \rangle \to \bC$ which has 
$\mu (1) = 1$, but is not assumed to be positive, hence does not 
necessarily belong to the space $\cD_r$ from (\ref{eqn:14a}).  In 
particular, for such $\mu$ we will still talk about its moments,
defined as in (\ref{eqn:14b}), and about its commuting e.g.f 
$f(z_1, \ldots , z_r)$, defined as in (\ref{eqn:14c}).
\end{remark}

An important example of functional in $\cD_r$ which we will 
use is the joint distribution of an $r$-tuple 
of $d \times d$ GUE random matrices with independent entries.

\begin{definition}   \label{def:16}
Let $M_1, \ldots , M_r$ be a tuple of $d \times d$ GUE random 
matrices with independent entries. 
We will denote by $\nurd$ the functional in $\cD_r$ which is
defined via the requirement that for every $k \in \bN$ and 
$( i(1), \ldots , i(k) ) \in \{ 1, \ldots , r \}^k$ one has
\begin{equation}   \label{eqn:16a}
\nurd (X_{i(1)} \cdots X_{i(k)}) =
( E \circ \tr_d) ( M_{i(1)} \cdots M_{i(k)}),
\end{equation}
where ``$( E \circ \tr_d)$'' is the expected normalized trace of 
a $d \times d$ random matrix.  The details of this definition are 
reviewed in Remark \ref{rem:52} below.
\end{definition}

\begin{theorem}   \label{thm:16}
Let $d \in \bN$ and consider the $*$-probability space 
$( \bC [ S_{\infty} ], \varphi_d )$, with $\varphi_d$ defined as in 
(\ref{eqn:1b}) above.  Let us also consider an $r \in \bN$ and an 
injective function 
\[
\ugamma : \{ 1, \ldots , r \} \times \bN \to 
\{ \gamma_n \mid n \in \bN \} \subseteq \bC [ S_{\infty} ],
\]
where $\{ \gamma_n \mid n \in \bN \}$ are the star-generators 
of $S_{\infty}$.  For every $p \in \{ 1, \ldots , r \}$ and 
$n \in \bN$ we consider the selfadjoint element 
\[
s_{p,n} := \frac{1}{\sqrt{n}} 
\Bigl( \ugamma (p,1) + \cdots + \ugamma (p,n) - \frac{n}{d} 
\Bigr)
\in \bC [ S_{\infty} ],
\]
then for every $n \in \bN$ we consider the joint distribution 
$\mu_n \in \cD_r$ of the $r$-tuple $s_{1,n}, \ldots , s_{r,n}$ 
(in the sense of Definition \ref{def:14}(3)).  One has that:

\vspace{6pt}

(1) The sequence $( \mu_n )_{n=1}^{\infty}$ 
defined above converges in moments to a limit $\murd \in \cD_r$.

\vspace{6pt}

(2) Let $f$ be the commuting e.g.f for the moments of $\murd$, 
and let $g$ be the commuting e.g.f for the moments of the joint 
GUE distribution $\nurd$ from Definition \ref{def:16}.  Then 
\begin{equation}   \label{eqn:16b}
f(z_1, \ldots , z_r) \cdot
e^{(z_1^2 + \cdots + z_r^2)/2d^2} = 
g(z_1, \ldots , z_r).
\end{equation}
\end{theorem}

$\ $

In the special case when $r=1$ and $\ugamma$ is defined by 
putting $\ugamma (1,n) := \gamma_n$, Theorem \ref{thm:16} 
retrieves the statement of Theorem \ref{thm:11}.  Indeed, 
in this case Equation (\ref{eqn:16b}) takes the form 
$f(z) \cdot e^{z^2/(2d^2)} = g(z)$, and we recognize the two 
sides of this equality to be the e.g.f's for moments of 
$\mu_d * N(0, 1/d^2)$ and respectively of $\nu_d$, with 
$\mu_d$ and $\nu_d$ as in the statement of Theorem \ref{thm:11}.  
It follows that $\mu_d * N(0, 1/d^2)$ and $\nu_d$ have the 
same sequence of moments.  It is easy to check that this 
sequence of moments grows slowly enough to ensure the 
uniqueness of the underlying probability distribution, and 
the equality $\mu_d * N(0, 1/d^2) = \nu_d$ hence follows.

$\ $

\subsection{In the background: functions on set-partitions.} 

$\ $

\noindent
We now move to explain how Theorem \ref{thm:16} comes about. 
In order to do so, it is relevant that we review some ideas 
concerning functions on set-partitions that were championed 
by Bozejko and Speicher in \cite{BoSp1996}.  A quick review 
of notation: in Definition \ref{def:17} below we consider 
partially ordered sets of the form ``$( \cP (k), \leq )$'' 
where $k \in \bN$ and $\cP (k)$ denotes the set of all 
partitions of $\{ 1, \ldots , k \}$, partially ordered by 
reverse refinement (for $\pi = \{ V_1, \ldots , V_p \}$ 
and $\rho = \{ W_1, \ldots, W_q \}$ in $\cP (k)$ we write 
``$\pi \leq \rho$'' to mean that every block $V_i$ of $\pi$
is contained in some block $W_j$ of $\rho$).  We let 
$\cP_2 (k)$ denote the set of all pair-partitions of 
$\{ 1, \ldots , k \}$: 
\[
\cP_2 (k) := \{  \pi \in \cP (k) 
\mid \mbox{every block $V$ of $\pi$ has $|V| = 2$} \}, 
\ \ k \in \bN , 
\]
with the convention that 
$\cP_2 (k) = \emptyset$ when $k$ is odd.  Moreover, we will 
use the standard notation that
\begin{equation}  \label{eqn:13x}
\left\{  \begin{array}{l}
\mbox{ for any $k \in \bN$ and any tuple 
$\uI = ( \uI (1), \ldots \uI (k) ) \in \bN^k$,}  \\
\mbox{ $\Ker ( \uI ) \in \cP (k)$ denotes the partition of 
$\{ 1, \ldots , k \}$ into level sets of $\uI$.}
\end{array}    \right.
\end{equation} 
This notation is useful when one deals with exchangeable 
sequences -- indeed, one of the possible descriptions (cf. 
Remark \ref{def:21} below) for the exchangeability of a sequence 
$\ans$ in a $*$-probability space $( \cA , \varphi )$ is that
\begin{equation}   \label{eqn:13y}
\left\{   \begin{array}{c}
\varphi (a_{\uI (1)} \cdots a_{ \uI (k)} ) 
         = \varphi (a_{\uJ (1)} \cdots a_{\uJ (k)} )      \\
\mbox{for every $k \in \bN$ and $\uI , \uJ  \in \bN^k $
      such that $\Ker ( \uI ) = \Ker ( \uJ )$.}
\end{array}   \right.
\end{equation}

$\ $

\begin{definition}   \label{def:17}
Let $( \cA , \varphi )$ be a $*$-probability space and let 
$\ans$ be an exchangeable sequence of selfadjoint elements 
of $\cA$.  

\vspace{6pt}

(1) The {\em function on partitions} associated to $\ans$ is
$\bt : \sqcup_{k=1}^{\infty} \cP (k) \to \bC$ defined as follows:
for every $k \in \bN$ and $\pi \in \cP (k)$ we put 
\begin{equation}   \label{eqn:17a}
\left\{  \begin{array}{l}
\bt ( \pi ) := \varphi \bigl(
a_{\uI (1)} , \ldots , a_{\uI (k)}  \bigr),
\mbox{ where $\uI \in \bN^k$ is}             \\
\mbox{ any $k$-tuple such that $\mathrm{Ker} ( \uI ) = \pi$. }
\end{array}    \right.
\end{equation}
This formula is unambiguous due to (\ref{eqn:13y}) above.

\vspace{6pt}

(2) Consider the algebra 
$\bC \langle \, \{ X_p \mid p \in \bN \} \, \rangle$ of 
polynomials in a countable collection of non-commuting indeterminates.
Following \cite{BoSp1996}, we use the function $\bt$ from (\ref{eqn:17a}) 
in order to define a linear functional 
$\mu : \bC \langle \, \{ X_p \mid p \in \bN \} \, \rangle \to \bC$
via the prescription that $\mu (1) = 1$ and that for every
$k \in \bN$ and every tuple 
$\uI = ( \uI (1), \ldots , \uI (k) ) \in \bN^k$ we put
\begin{equation}   \label{eqn:17b}
\mu \bigl( X_{\uI (1)} \cdots X_{\uI (k)} \bigr) =
\sum_{ \pi \in \cP_2 (k), \,  \pi \leq Ker \, ( \uI )}
\ \bt ( \pi ).
\end{equation}
In (\ref{eqn:17b}) we use the convention that 
$\mu ( X_{\uI(1)} \cdots X_{\uI (k)} ) = 0$ whenever
$\{ \pi \in \cP_2 (k) \mid \pi \leq \Ker ( \uI ) \}
= \emptyset$, i.e. whenever the partition $\Ker ( \uI )$ 
has at least one block of odd cardinality.
\end{definition}

\begin{remark}  \label{rem:18}
It was shown in \cite{BoSp1996} (see review in Section 2
below) that, in the framework of the preceding definition, 
the only additional requirement needed to ensure that the 
CLT works on the exchangeable sequence $\ans$ is that $\bt$ 
from (\ref{eqn:17a}) satisfies:
\begin{equation}   \label{eqn:18a}
\left\{  \begin{array}{l}
\bt ( \pi ) = 0 \ \mbox{ whenever the partition 
$\pi \in \sqcup_{k=1}^{\infty} \cP (k)$ }                \\
\mbox{ has at least one block $V$ with $|V| = 1$.}
\end{array}  \right.
\end{equation}
If this additional requirement is fulfilled, then for every 
$r \in \bN$ one gets 
a suitable ``exchangeable CLT for $r$-tuples'',
where the limit law is the restriction of the linear 
functional $\mu$ from (\ref{eqn:17b}) to the 
subalgebra $\bC \langle X_1, \ldots , X_r \rangle$ of
$\bC \langle \, \{ X_p \mid p \in \bN \} \, \rangle$.
The precise statement of this exchangeable CLT for 
$r$-tuples is reviewed in Theorem \ref{thm:25} below.
\end{remark}

In order to connect the notions introduced in Definition \ref{def:17} 
to the framework of the present paper, we make the following notation.

\begin{notation}   \label{def:19}
Let $d \in \bN$ and consider the $*$-probability space 
$( \bC [ S_{\infty} ], \varphi_d )$, with $\varphi_d$ defined as in
(\ref{eqn:1b}) above.

\vspace{6pt}

(1) Let $( \gamma_n )_{n=1}^{\infty}$ be the exchangeable sequence 
of star-generators from (\ref{eqn:1c}).  We denote the function on 
partitions associated to this exchangeable sequence by $\bu_d$, and 
we denote the corresponding linear functional on 
$\bC \langle \, \{ X_p \mid p \in \bN \} \, \rangle$ by
$\nu^{( \infty , d)}$.

\vspace{6pt}

(2) Consider on the other hand the exchangeable sequence 
$( \gamma_n - \frac{1}{d} )_{n=1}^{\infty}$ obtained by centering the 
star-generators $\gamma_n$ with respect to $\varphi_d$.  We denote the 
function on partitions associated to 
$( \gamma_n - \frac{1}{d} )_{n=1}^{\infty}$ by $\bt_d$, and
we denote the corresponding linear functional on 
$\bC \langle \, \{ X_p \mid p \in \bN \} \, \rangle$ by 
$\mu^{( \infty , d)}$.
\end{notation}

We can now put forward three separate propositions which, 
when combined together, give the statement of Theorem \ref{thm:16}.

\begin{proposition}  \label{prop:110}
Let $d, r$ be in $\bN$.

(1) The linear functional $\mu^{( \infty , d)} :
\bC \langle \, \{ X_p \mid p \in \bN \} \, \rangle \to \bC$ 
from Notation \ref{def:19}(2) is positive, that is, 
one has $\mu ( P^{*}P ) \geq 0$ for all 
$P \in \bC \langle \, \{ X_p \mid p \in \bN \} \, \rangle$.
Consequently, the restriction 
\begin{equation}  \label{eqn:110a}
\mu^{(r,d)} := \mu^{( \infty ,d )} \mid
\bC \langle X_1, \ldots , X_r \rangle
\end{equation}
belongs to the space of distributions $\cD_r$ introduced in
Definition \ref{def:14}(1).

\vspace{6pt}

(2) Consider the distributions $\mu_n \in \cD_r$ that are
indicated in the statement of Theorem \ref{thm:16}.  Then 
the sequence $( \mu_n )_{n=1}^{\infty}$ converges in moments 
to the distribution $\mu^{(r,d)}$ from (\ref{eqn:110a}). 
\end{proposition}

\begin{proposition}  \label{prop:111}
Let $d, r$ be in $\bN$.
The linear functional $\nu^{( \infty , d)}$ on
$\bC \langle  \{ X_p \mid p \in \bN \} \rangle$ 
which was introduced in Notation \ref{def:19}(1) is positive. 
The restriction of $\nu^{( \infty , d)}$ to the subalgebra 
$\bC \langle X_1, \ldots , X_r \rangle$ of 
$\bC \langle \, \{ X_p \mid p \in \bN \} \, \rangle$ 
is precisely equal to the distribution $\nu^{(r,d)} \in \cD_r$
that appeared (in connection to GUE matrices) in Definition
\ref{def:16}.
\end{proposition}

\begin{remark}  \label{rem:112}
Proposition \ref{prop:110} is a direct application of the exchangeable 
CLT theorem -- the only verifications that need to be done are
that the sequence of centered star-generators 
$( \gamma_n  - \frac{1}{d} )_{n=1}^{\infty}$ is indeed exchangeable
in $( \bC [ S_{\infty} ], \varphi_d )$, and that the associated  
function on partitions $\bt_d : \sqcup_{k=1}^{\infty} \cP (k) \to \bC$
satisfies the requirement (\ref{eqn:18a}) mentioned in Remark \ref{rem:18}.
We will make these verifications in Section 4 below.

Proposition \ref{prop:111} is of a different nature (has to be, 
since the function on partitions 
$\bu_d : \sqcup_{k=1}^{\infty} \cP (k) \to \bC$ associated to 
$( \gamma_n )_{n=1}^{\infty}$ clearly does not satisfy 
(\ref{eqn:18a})).  The proof of this proposition is
obtained by the direct examination of certain joint moments of the 
$\gamma_n$'s, which we then compare against the corresponding joint 
moments of a relevant exchangeable sequence of GUE matrices.
The review of GUE matrices and the proof of Proposition \ref{prop:111}
are made in Section 5 of the paper.

Upon invoking the interpretations of the distributions 
$\mu^{(r,d)},  \nu^{(r,d)}$ which follow from Propositions 
\ref{prop:110} and \ref{prop:111}, 
there is only one item remaining to be proved in Theorem 
\ref{thm:16}, namely Equation (\ref{eqn:16b}) in part (2) of 
the theorem.  This item is just a statement about what happens 
to commuting e.g.f's when we perform a translation of an 
exchangeable sequence, 
and is covered by the following general result.
\end{remark}

\begin{proposition}   \label{prop:113}
Let $( \cA , \varphi )$ be a $*$-probability space.
Let $\ans$ be an exchangeable sequence of selfadjoint 
elements of $\cA$, with associated function 
$\bt : \sqcup_{k=1}^{\infty} \cP (k) \to \bC$, such 
that $\bt$ fulfills the vanishing condition indicated 
in (\ref{eqn:18a}).  Pick a $\lambda \in \bR$ and put
\begin{equation}  \label{eqn:113a}
b_n := a_n + \lambda, \ \ n \in \bN .
\end{equation}
This creates a new exchangeable sequence $\bns$, which 
also has an associated function on partitions,
$\bu : \sqcup_{k=1}^{\infty} \cP (k) \to \bC$.  Let 
$\mu, \nu : \bC \langle \{ \, X_p \mid p \in \bN \} \, \rangle
\to \bC$ be the linear functionals obtained by using $\bt$ 
and $\bu$, respectively, via the recipe described in 
Definition \ref{def:17}(2).  Finally, pick an $r \in \bN$ 
and let $f,g$ denote the commuting e.g.f's of the functionals
$\mu \mid \bC \langle X_1, \ldots , X_r \rangle$ and
respectively
$\nu \mid \bC \langle X_1, \ldots , X_r \rangle$.
Then one has:
\begin{equation}  \label{eqn:113b}
f(z_1, \ldots , z_r) \cdot
e^{\lambda^2 (z_1^2 + \cdots + z_r^2)/2} = 
g(z_1, \ldots , z_r).
\end{equation}
\end{proposition}

$\ $

Proposition \ref{prop:113} can in particular be applied 
to the case when $a_n = \gamma_n - \frac{1}{d}$ in the 
$*$-probability space $( \bC [ S_{\infty} ], \varphi_d )$, 
with $\lambda = 1/d$, and this leads precisely to the 
statement of Theorem \ref{thm:16}(2).

$\ $

\subsection{Organization of the paper.}

$\ $

\noindent
In view of the discussion in Section 1.3, what we are left 
to do is show the proofs of the Propositions \ref{prop:110},
\ref{prop:111} and \ref{prop:113}.

Besides the present Introduction, the paper has four sections.
Section 2 is devoted to the review of the relevant facts related 
to exchangeable sequences, and of the exchangeable CLT theorem.
This is followed by the proof of Proposition \ref{prop:113},
which is done in Section 3.
In Section 4 we return to the framework of the $*$-probability space 
$( \bC [S_{\infty}] , \varphi_d )$, and we verify that the centered 
star-generators $( \gamma_n - \frac{1}{d} )_{n=1}^{\infty}$ have indeed 
the required properties which ensure the validity of Proposition 
\ref{prop:110}.  In the final Section 5 we establish the connection 
(formalized in Proposition \ref{prop:56}) between star-generators and
an exchangeable sequence of $d \times d$ GUE matrices with independent
entries, and we prove Proposition \ref{prop:111}.

$\ $

\section{Review of CLT for exchangeable sequences}

\begin{definition-and-remark}   \label{def:21}
We collect here a few basic notions and facts,
some of them already mentioned in the Introduction,
related to exchangeable sequences. 

Let $( \cA , \varphi )$ be a $*$-probability space and let 
$\ans$ be a sequence of selfadjoint elements of $\cA$.
Quantities of the form 
\[
\varphi (a_{\uI (1)} \cdots a_{\uI (k)} ), 
\mbox{ with $k \in \bN$ and $\uI  : \{ 1, \ldots , k \} \to \bN$}
\]
go under the name of {\em joint moments} of $\ans$.  We say that 
the sequence $\ans$ is {\em exchangeable} to mean that its joint 
moments are invariant under the natural action of $S_{\infty}$,  
that is:
\begin{equation}   \label{eqn:21a}
\left\{   \begin{array}{c}
\varphi (a_{\uI (1)} \cdots a_{ \uI (k)} ) 
         = \varphi (a_{\uJ (1)} \cdots a_{\uJ (k)} )      \\
\mbox{for every $k \in \bN$ and 
      $\uI , \uJ  : \{ 1, \ldots , k \} \to \bN$}         \\
\mbox{for which $\exists \, \tau \in S_{\infty}$
      such that $\uJ  = \tau \circ \uI$.}
\end{array}   \right.
\end{equation}

Recall from Equation (\ref{eqn:13x}) in the Introduction 
that, upon treating a tuple $\uI \in \bN^k$ as a function 
$\uI : \{ 1, \ldots , k \} \to \bN$, one defines
the kernel $\Ker ( \uI )$ as the partition of 
$\{ 1, \ldots , k \}$ into level-sets of $\uI$: two 
numbers $p,q \in \{ 1, \ldots , k \}$ belong 
to the same block of $\Ker (\uI)$ if and only if 
$\uI(p) = \uI(q)$.  It is easily seen that for two tuples
$\uI,\uJ : \{ 1, \ldots , k \} \to \bN$, the existence of 
a permutation $\tau \in S_{\infty}$ such that 
$\uJ = \tau \circ \uI$ is equivalent to the fact that
$\Ker (\uI) = \Ker (\uJ)$.   
Hence the definition of 
what it means for $\ans$ to be exchangeable can also be 
phrased in the way it was done in Equation (\ref{eqn:13y}) 
of the Introduction, which leads naturally to 
considering the function on partitions
$\bt : \sqcup_{k=1}^{\infty} \cP (k) \to \bC$ 
associated to $\ans$ in Definition \ref{def:17}(1).

Note that if $\ans$ is exchangeable, then the $a_n$'s are 
in particular {\em identically distributed}.  That is, for 
every $k \in \bN$ one has a common ``moment of order $k$'' 
for all the $a_n$'s:
\[
\varphi (a_1^k ) = \varphi (a_2^k ) = \cdots = \bt ( 1_k ),
\]
where, following standard combinatorics literature, we let 
$1_k \in \cP (k)$ denote the partition of $\{ 1, \ldots , k \}$
into only one block.
\end{definition-and-remark}

We now turn to the important vanishing condition that was 
mentioned in Equation (\ref{eqn:18a}) of Remark \ref{rem:18}.
For our purposes, it is convenient to phrase things in 
a way that also allows for translations of a given sequence 
of selfadjoint elements, and we thus proceed as follows.

\begin{definition}   \label{def:22}
Let $( \cA , \varphi )$ be a $*$-probability space and let 
$\ans$ be a sequence 
\footnote{The sequence $\ans$ considered in this definition 
is not assumed to be exchangeable.}
of selfadjoint 
elements of $\cA$.  We say that $\ans$ has the 
{\em singleton-factorization property} to mean that the 
following implication holds:
\begin{equation}   \label{eqn:22a}
\left\{   \begin{array}{l}
\mbox{Let } k \in \bN , \uI : \{ 1, \ldots , k \} \to \bN 
\mbox{ and $j_o \in \{ 1, \ldots , k \}$ be}                \\
\mbox{ $\ $ such that $\uI (j) \neq \uI (j_o)$ for all 
       $j \neq j_o$ in $\{ 1, \ldots , k \}$. }             \\
\mbox{Then }
\varphi ( a_{\uI (1)} \cdots a_{ \uI (k)} ) =  
\varphi ( a_{\uI (j_o)} ) \cdot 
\varphi \bigl( a_{\uI (1)} \cdots a_{\uI (j_o -1)} 
a_{\uI (j_o +1)} \cdots a_{\uI (k)} \bigr) .
\end{array}   \right.
\end{equation}
\end{definition}

The next proposition records the useful fact that both exchangeability
and singleton-factorization property are preserved when one does a 
translation of the sequence of $a_n$'s.  The proof of the proposition 
is straightforward (both statements (1) and (2) are verified via 
induction arguments on the length of the relevant joint moments), 
and is left as exercise to the reader.

\begin{proposition}  \label{prop:23}
Let $( \cA , \varphi )$ be a $*$-probability space, let $\ans$ 
be a sequence of selfadjoint elements of $\cA$, 
and let $\lambda$ be a real number.  We put 
$b_n := a_n + \lambda$, $n \in \bN$.

\vspace{6pt}

(1) Suppose that $\ans$ is exchangeable.  Then $\bns$ is 
exchangeable as well.

\vspace{6pt}

(2) Suppose that $\ans$ has the singleton-factorization property.  
Then $\bns$ has the singleton-factorization property as well.
\hfill  $\square$
\end{proposition}

\begin{remark}  \label{rem:24}
Let $( \cA , \varphi )$ be a $*$-probability space and let 
$\ans$ be a sequence of selfadjoint elements of $\cA$ 
such that:

\vspace{6pt}

(i) $\ans$ is exchangeable, and 

\vspace{6pt}

(ii) $\ans$ has the singleton-factorization property.

\vspace{6pt}

\noindent
Let $\bt : \sqcup_{k=1}^{\infty} \to \bC$ be the function on
partitions associated to $\ans$, in the way indicated in 
Definition \ref{def:17}(1).  The additional assumption (ii) 
gives a factorization formula satisfied by $\bt$: if 
$\pi \in \cP (k)$ has a block $V$ with $|V| = 1$ and if
$\pi_o \in \cP (k-1)$ is obtained from the restriction of 
$\pi$ to $\{ 1, \ldots , k \} \setminus V$ (by relabelling the 
elements of $\{ 1, \ldots , k \} \setminus V$ to become 
$1, \ldots , k-1$, in increasing order), then we have 
\begin{equation}   \label{eqn:24a}
\bt ( \pi ) = \bt ( 1_1 ) \cdot \bt ( \pi_o ) ,
\end{equation}
where $1_1$ is the unique partition in $\cP (1)$.
Repeated use of the formula (\ref{eqn:24a}) shows that, in 
this case, the function $\bt$ is completely determined if we
know its values on partitions $\pi$ which have no blocks of 
cardinality $1$ and if, in addition to that, we know what is 
$\bt ( 1_1 )$.

\vspace{6pt}

\noindent
[A concrete example: the common value of all the joint moments
$\varphi (a_i a_j a_i a_k )$ with $i \neq j \neq k \neq i$ 
in $\bN$ is recorded as $\bt ( \pi )$ for  $\pi = 
\{ \, \{ 1,3 \}, \, \{ 2 \} , \, \{ 4 \} \, \} \in \cP (4)$.
The singleton-factorization property ensures that this 
particular $\bt ( \pi )$ factors as 
$\bt ( 1_1 )^2 \cdot \bt (1_2)$, where $1_1$ is the unique 
partition of the set $\{ 1 \}$ and $1_2$ is the partition of 
$\{ 1,2 \}$ into a single block.]

\vspace{6pt}

Now, besides the assumptions (i) + (ii) above, suppose we also 
know that 

\vspace{6pt}

(iii) the  $a_n$'s are centered (we have $\varphi (a_n) = 0$ 
for all $n \in \bN$).

\vspace{6pt}

\noindent
Clearly, the assumption (iii) amounts to the fact that 
$\bt ( 1_1 ) =0$.  It is obvious that, in this case, the 
factorization rule from Equation (\ref{eqn:24a}) becomes 
precisely the vanishing condition that had been indicated 
in Equation (\ref{eqn:18a}) of Remark \ref{rem:18}.
\end{remark}

$\ $

We next review the limit theorem for exchangeable sequences 
that we want to use.  We follow the version that is indicated 
in \cite{BoSp1996} -- see Theorem 0 on page 138 and also 
Theorem 2 on page 142 of that paper.

\begin{theorem}   \label{thm:25}
(CLT for an exchangeable sequence, following \cite{BoSp1996}.)

\noindent
Let $( \cA , \varphi )$ be a $*$-probability space and let
$\ans$ be a sequence of selfadjoint elements of $\cA$ which 
satisfies the assumptions (i), (ii) and (iii) listed in the 
preceding remark.
Let $\bt : \sqcup_{k=1}^{\infty} \cP (k) \to \bC$ be the 
function on partitions associated to $\ans$ as in Definition
\ref{def:17}(1), and let $\mu : \bC 
\langle \, \{ X_p \mid p \in \bN \} \, \rangle \to \bC$
be the linear functional constructed from $\bt$ in the 
way described in Definition \ref{def:17}(2).  Then:

\vspace{6pt}

(1) $\mu$ is positive (that is, $\mu (P^{*}P) \geq 0$ for all 
$P \in \{ X_p \mid p \in \bN \} \, \rangle$).

\vspace{6pt}

(2) Consider an $r \in \bN$ and an injective function 
$\lambda : \{ 1, \ldots , r \} \times \bN \to  \bN$. 
We denote $\ua (p,n) := a_{\lambda (p,n)}$, for all 
$p \in \{ 1, \ldots , r \}$ and $n \in \bN$.  Moreover,
for every $p \in \{ 1, \ldots , r \}$ and $n \in \bN$ we put
\[
s_{p,n} := \frac{1}{\sqrt{n}} 
\Bigl( \ua (p,1) + \cdots + \ua (p,n) \Bigr) \in \cA ,
\]
and for every $n \in \bN$ we consider the joint distribution 
$\mu_n \in \cD_r$ of the $r$-tuple $s_{1,n}, \ldots , s_{r,n}$ 
(in the sense of Definition \ref{def:14}(3)).
Then the sequence $( \mu_n )_{n=1}^{\infty}$ converges in 
moments to the restriction  of the linear functional $\mu$ 
from part (1) to the subalgebra 
$\bC \langle X_1, \ldots , X_r \rangle$ of 
$\bC \langle \, \{ X_p \mid p \in \bN \} \, \rangle$.
\hfill $\square$
\end{theorem}

\begin{remark}   \label{rem:26}
Let $( \cA , \varphi )$ be a $*$-probability space and let
$\ans$ be a sequence of selfadjoint elements of $\cA$ which 
satisfies assumptions (i) and (ii) from Remark \ref{rem:24}, 
but where we do not make the assumption (iii) that the $a_n$'s 
are centered.  Denoting 
$\alpha_1 := \varphi (a_1) = \varphi (a_2) = \cdots \, ,$
we can still apply Theorem \ref{thm:25} to the sequence 
$( \abubble_n )_{n=1}^{\infty}$ where 
$\abubble_n := a_n - \alpha_1, \ n \in \bN$.
Indeed, the sequence $( \abubble_n )_{n=1}^{\infty}$ still 
satisfies (i) and (ii), by Proposition \ref{prop:23}, and in 
addition to that the $\abubble_n$'s are centered as well.  Note 
that in this case, if we want to write the elements $s_{p,n}$ 
involved in Theorem \ref{thm:25} directly in terms of the $a_n$'s, 
then they are
\[
s_{p,n} = \frac{1}{\sqrt{n}} 
\Bigl( \ua (p,1) + \cdots + \ua (p,n)  
- n \alpha_1 \Bigr),
\mbox{ for $p \in \{ 1, \ldots , r \}$ and $n \in \bN$.}
\]
\end{remark}

$\ $

\section{Translation of an exchangeable sequence, and proof of 
Proposition \ref{prop:113} }

We begin with a lemma which is, essentially, the special 
case ``$r=1$'' of the proposition we are aiming to prove.

\begin{lemma}   \label{lemma:31}
Let $( \cA , \varphi )$ be a $*$-probability space.
Let $\ans$ be an exchangeable sequence of selfadjoint 
elements of $\cA$, with associated function 
$\bt : \sqcup_{k=1}^{\infty} \cP (k) \to \bC$, and 
suppose that $\bt$ fulfills the vanishing condition 
indicated in (\ref{eqn:18a}) of Remark \ref{rem:18}.  
Pick a $\lambda \in \bR$ and put
\[
b_n := a_n + \lambda, \ \ n \in \bN .
\]
This creates a new exchangeable sequence $\bns$, which 
also has an associated function on partitions,
$\bu : \sqcup_{k=1}^{\infty} \cP (k) \to \bC$. 

Consider the power series 
$\widetilde{f}_1, \widetilde{g}_1 \in \bC [[z]]$
defined by
\begin{equation}   \label{eqn:31b}
\widetilde{f}_1 (z) = \sum_{m=0}^{\infty} 
\frac{ \alpha_{2m} }{ (2m)! } z^m \mbox{ and }
\widetilde{g}_1 (z) = \sum_{m=0}^{\infty} 
\frac{ \beta_{2m} }{ (2m)! } z^m,
\end{equation}
where for every $m \in \bN$ we put 
\[
\alpha_{2m} := \sum_{\pi \in \cP_2 (2m)} \bt (\pi)
\ \mbox{ and } \ \beta_{2m} :=
\sum_{\pi \in \cP_2 (2m)} \bu (\pi),
\]
and we also make the convention to put 
$\alpha_0 = \beta_0 = 1$.
Then one has
\begin{equation}   \label{eqn:31c}
\widetilde{f}_1 (z) \cdot e^{\lambda^2 z/2} 
= \widetilde{g}_1 (z).
\end{equation} 
\end{lemma}

\begin{proof}  The series appearing on both sides of (\ref{eqn:31c}) 
have constant terms equal to $1$, while for an $m \in \bN$ the 
equality of their coefficients of order $m$ amounts to
\begin{equation}   \label{eqn:31d}
\frac{\beta_{2m}}{ (2m)! } = 
\sum_{\ell = 0}^m 
\frac{\alpha_{2 \ell}}{ (2 \ell)! } \cdot 
\frac{ (\lambda^2 /2)^{m - \ell} }{ (m- \ell )!} 
\end{equation}
(where on the right-hand side we considered the Cauchy 
product of $\widetilde{f}_1 (z)$ with the series expansion for 
$e^{\lambda^2 z/2}$).  We fix for the whole proof an $m \in \bN$ 
for which we will verify that (\ref{eqn:31d}) holds.

Let $\pi$ be in $\cP_2 (2m)$ and let us 
consider a tuple
$\uI = ( \uI (1), \ldots , \uI (2m) ) \in \bN^{2m}$
such that $\Ker ( \uI ) = \pi$.  Then 
\begin{align*}
\bu ( \pi ) 
& = \varphi \bigl( b_{\uI (1)} \cdots b_{\uI (2m)} \bigr)   
  = \varphi \bigl( \ (a_{\uI (1)} + \lambda) \cdots 
                     (a_{\uI (2m)} + \lambda) \, \bigr)     \\
& = \sum_{A \subseteq \{ 1, \ldots , 2m \}} 
    \ \lambda^{2m - |A|} \cdot \varphi  
    \Bigl( \, \prod_{p \in A} \, a_{\uI (p)} \, \Bigr)     
  = \lambda^{2m} + 
    \sum_{ \emptyset \neq A \subseteq \{ 1, \ldots , 2m \}} 
    \ \lambda^{2m - |A|} \cdot \bt ( \pi | A),
\end{align*}
where for a non-empty set $A \subseteq \{ 1, \ldots , 2m \}$ we have 
denoted as ``$\pi | A$'' the partition in $\cP ( |A| )$ which is 
obtained by restricting 
\footnote{The restriction of $\pi$ to $A$ is the partition of 
$A$ into the blocks
$\{ A \cap V \mid V \in \pi \mbox{ and } A \cap V \neq \emptyset \}$.}
$\pi$ to $A$ 
and then by redenoting the elements 
of $A$ as $\{ 1, \ldots , |A| \}$, in increasing order.  We note that 
for every $\emptyset \neq A \subseteq \{ 1, \ldots , 2m \}$, all the 
blocks of $\pi | A$ have cardinality 1 or 2; and moreover, if at 
least one of those blocks has cardinality 1 then the hypothesis 
(\ref{eqn:18a}) assumed on $\bt$ entails that $\bt ( \pi | A ) = 0$.
Observe, moreover, that the requirement on $\pi | A$ to only 
have blocks of cardinality 2 is equivalent to the requirement 
that $A$ is a union of blocks of $\pi$; we will refer to this 
situation by saying that ``$A$ is a $\pi$-saturated set''.

The conclusion of the preceding paragraph is that for any 
$\pi \in \cP_2 (2m)$ we have 
\begin{equation}   \label{eqn:31e}
\bu ( \pi ) = \lambda^{2m} + 
\sum_{ \begin{array}{c}
{\scriptstyle \emptyset \neq A \subseteq \{ 1, \ldots , 2m \},}  \\
{\scriptstyle \pi-saturated}
\end{array} } 
\ \lambda^{2m - |A|} \cdot \bt ( \pi | A).
\end{equation}
In this formula we sum over $\pi \in \cP_2 ( 2m )$ and then we reverse
the order of summation in the ensuing double sum, to obtain:
\begin{equation}   \label{eqn:31f}
\beta_{2m} = \lambda^{2m} \cdot | \cP_2 (2m) | + 
\sum_{\emptyset \neq A \subseteq \{ 1, \ldots , 2m \}}
\lambda^{2m - |A|} \cdot 
\Bigl( \, \sum_{ \begin{array}{c}
{\scriptstyle \pi \in \cP_2 (2m),} \\
{\scriptstyle A \ is \ \pi-saturated} 
\end{array} } 
\ \bt ( \pi | A) \, \Bigr) .
\end{equation}

Now let us fix for a moment a non-empty set 
$A \subseteq \{ 1, \ldots , 2m \}$.  If $|A|$ is odd then, clearly, 
there are no partitions $\pi \in \cP_2 (2m)$ such that $A$ is 
$\pi$-saturated.  If $|A| = 2 \ell$ for some 
$\ell \in \{ 1, \ldots , m \}$, then we claim that:
\begin{equation}   \label{eqn:31g}
\sum_{ \begin{array}{c}
{\scriptstyle \pi \in \cP_2 (2m),} \\
{\scriptstyle A \ is \ \pi-saturated} 
\end{array} } 
\ \bt ( \pi | A) =
| \cP_2 ( 2m - 2 \ell ) | \cdot 
\sum_{\rho \in \cP_2 ( 2 \ell )} \bt ( \rho )
= | \cP_2 ( 2m - 2 \ell ) | \cdot \alpha_{2 \ell}.
\end{equation}
The formula (\ref{eqn:31g}) is easily verified upon noticing 
that the pair-partitions $\pi \in \cP_2 (2m)$ such that $A$ is 
$\pi$-saturated are parametrized by couples 
$( \rho , \rho ' ) \in \cP_2 ( 2 \ell ) \times 
\cP_2 ( 2 m - 2 \ell )$, where $\rho$ is turned into a 
pair-partition of $A$ while $\rho '$ is turned into
a pair-partition of $\{ 1, \ldots , 2m \} \setminus A$,
in the natural way.  

By plugging (\ref{eqn:31g}) into (\ref{eqn:31f}), 
we find that:
\begin{equation}   \label{eqn:31h}
\beta_{2m} = \lambda^{2m} \cdot | \cP_2 (2m) | + 
\sum_{\ell = 1}^m 
\left(  \begin{array}{c} 2m \\ 2 \ell \end{array} \right)
\cdot | \cP_2 ( 2m - 2 \ell ) | \cdot \alpha_{2 \ell}.
\end{equation}
In the latter equation we substitute 
$| \cP_2 (2m) | = (2m)!/( 2^m \, m!)$ and 
$| \cP_2 (2m - 2 \ell ) |$ = 

\noindent
$(2m - 2 \ell )!/( 2^{m- \ell} \, (m- \ell )!)$, and 
then 
some straightforward algebra leads to the required 
formula (\ref{eqn:31d}).
\end{proof}

$\ $

The next lemma provides a reduction from multi-variate case 
to univariate case in the commuting e.g.f's that are of 
interest in Proposition \ref{prop:113}.

\begin{lemma}   \label{lemma:32}
Suppose we are given a function
$\bt : \sqcup_{k=1}^{\infty} \cP (k) \to \bC$,
and we use it to construct a linear functional 
$\mu : \bC \langle \, \{ X_p \mid p \in \bN \} \, \rangle 
\to \bC$ via the recipe described in Definition \ref{def:17}(2).
Consider moreover an $r \in \bN$, and let 
$f_r (z_1, \ldots , z_r)$ be the commuting e.g.f (as in 
Definition \ref{def:14}(2)) for the moments of the restriction 
of $\mu$ to $\bC \langle X_1, \ldots , X_r \rangle \subseteq 
\bC \langle \, \{ X_p \mid p \in \bN \} \, \rangle$.

On the other hand, consider the series of 
one variable
$\widetilde{f}_1 (z) := \sum_{m=0}^{\infty} 
\frac{\alpha_{2m}}{ (2m)! } z^m$, where 
\[
\alpha_{2m} := \sum_{\pi \in \cP_2 (2m)} \bt ( \pi )
\mbox{ for $m \in \bN$, and we also put } \alpha_0 := 1.
\]

Then the series $f_r$ and $\widetilde{f}_1$ are related by 
the formula
\begin{equation}   \label{eqn:32b}
f_r ( z_1, \ldots , z_r ) =
\widetilde{f}_1 ( z_1^2 + \cdots + z_r^2 ).
\end{equation} 
\end{lemma}

\begin{proof}
We consider some integers $\ell_1 , \ldots , \ell_r \geq 0$
for which we will examine the coefficient of 
$z_1^{\ell_1} \cdots z_r^{\ell_r}$ in the series $f_r$.  Without 
loss of generality, we assume that not all of 
$\ell_1, \ldots , \ell_r$ are equal to $0$ (since the series 
on both sides of (\ref{eqn:32b}) have constant terms equal to 
$1$).  We denote $\ell_1 + \cdots + \ell_r =: k \in \bN$.

By direct inspection of the formula which defines the moments a 
commuting e.g.f, and then by invoking 
the specific formula for the moments of $\mu$, we see that: 
\[
\left(  \begin{array}{c}
\mbox{coefficient of}  \\
\mbox{$z_1^{\ell_1} \cdots z_r^{\ell_r}$ in $f_r$}
\end{array}  \right)
\ = \ \frac{1}{k!} \cdot
\sum_{ \begin{array}{c}
{\scriptstyle \uI = 
       ( \uI (1), \ldots \uI (k) ) \in \{ 1, \ldots , r \}^k } \\
{\scriptstyle with \ | \uI^{-1}(1)| = \ell_1, \ldots ,
                     | \uI^{-1}(r)| = \ell_r }
\end{array} } \ \ \mu \bigl( X_{\uI (1)} \cdots X_{\uI (k)} \bigr)
\]
\[
= \frac{1}{k!} \sum_{ \begin{array}{c}
{\scriptstyle \uI = 
       ( \uI (1), \ldots \uI (k) ) \in \{ 1, \ldots , r \}^k } \\
{\scriptstyle with \ | \uI^{-1}(1)| = \ell_1, \ldots ,
                     | \uI^{-1}(r)| = \ell_r }
\end{array} } \ \ 
\sum_{\pi \in \cP_2 (k), \,  \pi \leq \, Ker ( \uI )}
\ \bt (\pi).
\]
In the latter expression we reverse the order of the two
sums, and continue with
\begin{equation}   \label{eqn:32c} 
= \frac{1}{k!} \sum_{\pi \in \cP_2 (k)} \bt (\pi) \cdot
\, \vline \ \Bigl\{ \uI \in \{1, \ldots , r \}^k 
\begin{array}{ll}
\vline & | \uI^{-1}(1)| = \ell_1, \ldots , | \uI^{-1}(r)| = \ell_r,  \\
\vline & \mbox{and } \Ker ( \uI ) \geq \pi 
\end{array}  \Bigr\} \ \vline .
\end{equation}

If not true that all of $\ell_1, \ldots , \ell_r$ are even, 
then it is immediate that
\[
\Bigl\{ \uI \in \{1, \ldots , r \}^k 
\begin{array}{ll}
\vline & | \uI^{-1}(1)| = \ell_1, \ldots , | \uI^{-1}(r)| = \ell_r,  \\
\vline & \mbox{and } \Ker ( \uI ) \geq \pi 
\end{array}  \Bigr\}  = \emptyset, 
\mbox{ for all $\pi \in \cP_2 (k)$,}  
\]
and Equation (\ref{eqn:32c}) entails that the coefficient 
of $z_1^{\ell_1} \cdots z_r^{\ell_r}$ in $f_r$ is equal to $0$. 

Let us now assume that $\ell_1, \ldots , \ell_r$ are even, and let us 
write $\ell_1 = 2j_1, \ldots , \ell_r = 2j_r$.  We also denote 
$j_1 + \cdots + j_r =: m$ (hence $m = k/2 \in \bN$).  Observe that 
for every $\pi \in \cP_2 (k)$, we have 
\begin{equation}   \label{eqn:32d}
\vline \ \Bigl\{ \uI \in \{1, \ldots , r \}^k 
\begin{array}{ll}
\vline & | \uI^{-1}(1)| = \ell_1, \ldots , | \uI^{-1}(r)| = \ell_r,  \\
\vline & \mbox{and } \Ker ( \uI ) \geq \pi 
\end{array}  \Bigr\}  \ \vline \ = 
\frac{m!}{j_1! \cdots j_r!}.
\end{equation}
Indeed, the tuples $\uI$ which are being counted on the left-hand side
of (\ref{eqn:32d}) are constant along every block of $\pi$, and
can be thought of as ``colourings'' of the blocks of $\pi$, 
where $j_1$ blocks have colour $1, \ldots, j_r$ blocks have colour
$r$; the number of such colourings is given by the multinomial 
coefficient indicated on the right-hand side of (\ref{eqn:32d}).
Upon plugging (\ref{eqn:32d}) into (\ref{eqn:32c}), we conclude that 
the coefficient of $z_1^{2j_1} \cdots z_r^{2j_r}$ in $f_r$ is equal to 
\begin{equation}   \label{eqn:32e}
\frac{1}{(2m)!} \cdot \frac{ m! }{j_1! j_2! \cdots j_r!}
\cdot \sum_{\pi \in \cP_2 (2m)} \bt ( \pi ),
\ \mbox{ hence to }
\frac{\alpha_{2m}}{ (2m)! }
\cdot \frac{ m! }{j_1! j_2! \cdots j_r!} .
\end{equation} 

It is immediately verified that the coefficients of the series 
\[
\widetilde{f}_1 (z_1^2 + \cdots + z_r^2) := 
1 + \sum_{m=1}^{\infty} 
\frac{\alpha_{2m}}{ (2m)! } (z_1^2 + \cdots + z_r^2)^m
\]
have exactly the same values as those found above for the coefficients 
of $f_r$ (namely the coefficients of monomials 
$z_1^{2j_1} \cdots z_r^{2j_r}$ are as shown in (\ref{eqn:32e}), 
while all the other coefficients are equal to $0$).  This 
concludes the verification of the required equality
(\ref{eqn:32b}).
\end{proof}

$\ $

\begin{ad-hoc-item}
{\bf Proof of Proposition \ref{prop:113}.}
The series $f(z_1, \ldots , z_r)$ in Proposition \ref{prop:113}
is the same as the $f_r (z_1, \ldots, z_r)$ from the 
above Lemma \ref{lemma:32}, and can therefore be expressed
as $\widetilde{f}_1 (z_1^2 + \cdots + z_r^2)$, with 
$\widetilde{f}_1 \in \bC [[z]]$ as described in 
Lemma \ref{lemma:32}.  Likewise, the series 
$g(z_1, \ldots , z_r)$ in Proposition \ref{prop:113} can 
be written as ``$g_r (z_1, \ldots , z_r)$'' in the framework of 
Lemma \ref{lemma:32}, where we now start from the 
function on partitions $\bu$ (instead of $\bt$); consequently,
we have the formula
\[
g(z_1, \ldots , z_r) = 
\widetilde{g}_1 (z_1^2 + \cdots + z_r^2),
\]
for the corresponding series $\widetilde{g}_1 \in \bC [[z]]$
defined in the way described in Lemma \ref{lemma:32}.  We are
only left to invoke Lemma \ref{lemma:31} which connects the series
$\widetilde{f}_1$ and $\widetilde{g}_1$ -- the formula 
(\ref{eqn:31c}) from Lemma \ref{lemma:31} converts precisely into
the formula (\ref{eqn:113b}) from Proposition \ref{prop:113}.
\hfill $\square$
\end{ad-hoc-item}

$\ $

\section{The exchangeable sequence of star-generators
of $S_{\infty}$, and proof of Proposition \ref{prop:110}}

\begin{notation-and-remark}  \label{def:41}

(1) Let $S_{\infty}$ be the infinite symmetric group 
(as in (\ref{eqn:1a}) of the Introduction).  We will 
write the permutations in $S_{\infty}$ by using cycle 
notation, where we only indicate the cycles of length 
$\geq 2$ of the permutation -- it is implicitly assumed 
that all the numbers in $\bN$ that are not indicated 
in the cycle notation are fixed points 
of the permutation in question.  This convention was in
particular used in Equation (\ref{eqn:1c}) of the 
Introduction, where we considered the star-generators
\[
\gamma_n := (1, n+1), \ \ n \in \bN .
\]

\vspace{6pt}

(2) As mentioned in the Introduction, we will use the 
notation $|| \tau ||$ for the minimal number of factors
required in a factorization of $\tau$ into transpositions
and where, by convention, we have $|| \tau || = 0$ if and 
only if $\tau$ is the identity permutation of $\bN$.
Note that, as a consequence of the fact that the set 
of transpositions in $S_{\infty}$ is invariant under 
conjugation, the map $\tau \mapsto || \tau ||$ is constant 
on conjugacy classes of $S_{\infty}$.

\vspace{6pt}

(3) We will use the notation ``$\#$'' for the number of 
cycles (including fixed points) of a permutation 
$\tau \in S_{\infty}$ on a given invariant finite set.  
More precisely: if $\tau \in S_{\infty}$ and if 
$A \subseteq \bN$ is a finite set such that 
$\tau (A) = A$, then we denote
\[
\# ( \tau \mid A )
:= \bigl(  \mbox{
number of orbits into which $A$ is partitioned by 
the action of $\tau$} \, \bigr) .
\]
This notation is useful for giving an
alternative description of the number $|| \tau ||$ reviewed 
in (2) above; indeed, it is easy to verify that one has the 
formula 
\begin{equation}   \label{eqn:41c}
|| \tau || = |A| - \# ( \tau \mid A ),
\end{equation}
holding for $\tau \in S_{\infty}$ and with $A$ being any finite 
subset of $\bN$ such that $\tau (b) = b$ for all 
$b \in \bN \setminus A$.  

\vspace{6pt}

(4) On the group algebra $\bC [ S_{\infty} ]$ we consider the 
$*$-operation determined by the requirement that 
\[
\tau^{*} := \tau^{-1}, \ \ \forall \, \tau \in S_{\infty}.
\]
That is, every permutation $\tau$ in $S_{\infty}$ becomes a
unitary element of $\bC [ S_{\infty} ]$.  Note that if $\tau$ 
is a product of disjoint transpositions (hence $\tau = \tau^{-1}$),
then $\tau$ is at the same time a selfadjoint element of 
$\bC [ S_{\infty} ]$.  In particular, the star-generators 
$( \gamma_n )_{n=1}^{\infty}$ form a sequence of selfadjoint 
elements of $\bC [S_{\infty}]$.

\vspace{6pt}

(5) Let $d$ be in $\bN$ and let 
$\varphi_d : \bC [ S_{\infty} ] \to \bC$ be the linear 
functional defined in the way indicated in Equation (\ref{eqn:1b})
of the Introduction.  It is obvious that $\varphi_d (1) = 1$, 
and it turns out that, moreover, $\varphi_d$ is positive; hence
$( \bC [ S_{\infty} ], \varphi_d )$ is a $*$-probability space.
The positivity property of $\varphi_d$ can be directly verified 
by examining the action of permutations on words of finite length 
over the alphabet $\{ 1, \ldots , d \}$; for a detailed presentation 
of how this goes, we refer the reader to \cite{GnGoKe2013}.  

We also note that $\varphi_d$ has the trace property: 
\begin{equation}   \label{eqn:41d}
\varphi_d (ab) = \varphi_d (ba), \ \ \forall \, 
a, b \in \bC [S_{\infty}]. 
\end{equation}
Indeed, this boils down to checking that 
$|| \sigma \tau || = || \tau \sigma ||$ for all 
$\sigma, \tau \in S_{\infty}$, and the latter equality follows from
the fact that $\sigma \tau$ and $\tau \sigma$ belong to the same 
conjugacy class of $S_{\infty}$.  In connection to the trace 
property, we mention that the restriction of $\varphi_d$ to 
the group $S_{\infty}$ is what is called an ``extremal character'' 
of this group.  In the well-known parametrization of Thoma 
for such characters (see e.g.  Section 3 of the survey paper 
\cite{Ol2003}, or Section 4.1 of the monograph \cite{BoOl2016}), 
$\varphi_d$ is the character of $S_{\infty}$ parametrized by the Thoma 
double sequence $( \alpha_n ; \beta_n )_{n=1}^{\infty}$ having
\[
\alpha_1 = \cdots = \alpha_d = 1/d, 
\ \alpha_n = 0 \mbox{ for $n>d$, and $\beta_n = 0$ for all $n \in \bN$}.
\]
This identification as an extreme character can also be used as an 
argument for the fact that $\varphi_d$ is a positive functional 
on $\bC [ S_{\infty} ]$.
\end{notation-and-remark}

\begin{remark}   \label{rem:43}
In connection to the framework from Notation \ref{def:41}(5),
one may wonder what is so special about using a base of the form 
$1/d$, with $d \in \bN$, in the formula (\ref{eqn:1b}) which defines
the linear functional considered on $\bC [ S_{\infty} ]$.  Why 
doesn't one consider a linear functional 
$\varphi : \bC [S_{\infty}] \to \bC$ defined by the requirement that
\begin{equation}   \label{eqn:43a}
\varphi ( \tau ) = q^{ || \tau || }, 
\ \ \forall \, \tau \in S_{\infty},
\end{equation}
where $q$ is some arbitrary (but fixed) real number?  The reason 
is that the functional defined by Equation (\ref{eqn:43a}) is 
positive on $\bC [ S_{\infty} ]$ if and only if $q$ belongs to the 
special subset
\begin{equation}   \label{eqn:43c}
\{ \frac{1}{d} \mid d \in \bN \} \cup \{ 0 \} \cup
\{ - \frac{1}{d} \mid d \in \bN \} \subseteq \bR ;
\end{equation}
so the choice $q = 1/d$ is, in fact, not too restrictive.  
Concerning the 
possible alternative of using $q = -1/d$, let us observe that the 
functionals corresponding to $q = 1/d$ and $q = - 1/d$ in (\ref{eqn:43a}) 
only differ by a multiplication with the character 
$\tau \mapsto \mbox{sign} (\tau)$ on $S_{\infty}$; as a consequence of 
this fact, the whole layout of the present paper wouldn't change much if 
we would choose to work with $q = -1/d$ instead of $q = 1/d$.  We also 
note that the choice $q = 0$ in the set of possible values of $q$ 
indicated in (\ref{eqn:43c}) would make the functional $\varphi$ defined 
by (\ref{eqn:43a}) become the so-called ``canonical trace'' associated 
to the regular representation of $S_{\infty}$; in this case, as mentioned 
in the Introduction, the counterpart of our main Theorem \ref{thm:16} is 
a result obtained by Biane in \cite{Bi1995}.

In order to prove that the positivity requirement restricts $q$ to the 
set of values indicated in (\ref{eqn:43c}), one can use a direct argument, 
described as follows.  The set of values 
\begin{equation}   \label{eqn:43d}
\Bigl\{ q  \in \bR 
\begin{array}{lc}
\vline & \mbox{ the functional $\varphi$ defined in (\ref{eqn:43a})}  \\
\vline & \mbox{ is positive on $\bC [ S_{\infty} ]$ }
\end{array}  \Bigr\}
\end{equation}
is symmetric, because the functionals ``$\varphi$'' corresponding to $q$ 
and to $-q$ are obtained from each other by multiplication with the sign 
character of $S_{\infty}$.  It thus suffices to look at a $q$ from the set 
(\ref{eqn:43d}) such that $q < 0$.  For every $n \in \bN$, let 
$S_n := \{ \tau \in S_{\infty} \mid \tau (k) = k \mbox{ for all } k > n \}$, 
and consider the square matrix (of size $n!$) 
$G_n := [ \varphi ( \sigma^{-1} \tau ) ]_{\sigma, \tau \in S_n}$.
It is easily seen that, since $\varphi$ is positive, the matrices 
$G_n$ have to be non-negative definite for all $n \in \bN$.  But it is 
also easily seen that, for every $n \in \bN$, $G_n$ has the eigenvalue
\[
\lambda_n = \sum_{ \tau \in S_n }
\ q^{ || \tau || }
\]
(corresponding to the eigenvector that has all the components equal 
to $1$), and that this eigenvalue can be factored as
$\lambda_n = (1+q)(1+2q) \cdots (1+ (n-1)q)$.
Finally, for our $q < 0$, the condition that $\lambda_n \geq 0$ for 
all $n \in \bN$ forces $q$ to be of the form $q = - 1/d$ with 
$d \in \bN$. 
\end{remark}

$\ $

In the framework from Notation \ref{def:41}, we now consider 
the sequence $( \gamma_n )_{n=1}^{\infty}$ of star-generators of 
$S_{\infty}$, and we prove that it satisfies the hypotheses 
discussed in Theorem \ref{thm:25} and Remark \ref{rem:26}.

\begin{proposition}   \label{prop:44}
Let $d$ be in $\bN$.
Consider the $*$-probability space
$( \, \bC [ S_{\infty} ], \varphi_d )$ and the sequence of selfadjoint 
elements $( \gamma_n )_{n=1}^{\infty}$ in $\bC [ S_{\infty} ]$.
Then:

\vspace{6pt}

(1) $( \gamma_n )_{n=1}^{\infty}$ is exchangeable.

\vspace{6pt}

(2) $( \gamma_n )_{n=1}^{\infty}$ has the singleton-factorization 
property.
\end{proposition}

\begin{proof}  
(1) Consider, same as in Definition \ref{def:21}: a $k \in \bN$
and two tuples $\uI, \uJ : \{ 1, \ldots , k \} \to \bN$ for which there
exists a permutation $\tau \in S_{\infty}$ such that 
$\uJ = \tau \circ \uI$.
We have to verify the equality
\begin{equation}  \label{eqn:44a}
\varphi_d ( \gamma_{\uI(1)} \cdots \gamma_{\uI(k)} ) =  
\varphi_d ( \gamma_{\uJ(1)} \cdots \gamma_{\uJ(k)} ).
\end{equation}  

Let $\theta : \bN \to \bN$ be defined by putting $\theta (1) = 1$ 
and $\theta (n) = 1 + \tau (n-1)$ for $n \geq 2$.  It is immediate 
that $\theta \in S_{\infty}$.  For every $1 \leq h \leq k$ we have:
\begin{equation}  \label{eqn:44b}
\theta \gamma_{\uI (h)} \theta^{-1}
= \theta \, ( 1, \uI (h) + 1) \, \theta^{-1}
= ( 1, \uJ (h) + 1) = \gamma_{\uJ(h)},
\end{equation}  
where at the second equality sign we used the fact that $\theta (1) = 1$ 
and $\theta (\uI(h) + 1) = 1 + \tau (\uI(h)) = 1 + \uJ(h)$. 
From (\ref{eqn:44b}) we infer that 
\[
\theta \bigl( \gamma_{\uI(1)} \cdots \gamma_{\uI(k)} \bigr)    
\theta^{-1} = \gamma_{\uJ(1)} \cdots \gamma_{\uJ(k)} \in \bC [S_{\infty}], 
\]
and the required equality (\ref{eqn:44a}) follows from the fact that 
$\varphi_d$ has the trace property.

\vspace{6pt}

(2) Consider, same as in Definition \ref{def:22}: a $k \in \bN$, 
a tuple $\uI : \{ 1, \ldots , k \} \to \bN$ and an index
$j \in \{ 1, \ldots , k \}$ such that $\uI(j) \neq \uI( \ell )$ for all 
$\ell \neq j$ in $\{ 1, \ldots , k \}$.  We have to verify the equality
\begin{equation}  \label{eqn:44c}
\varphi_d ( \gamma_{\uI(1)} \cdots \gamma_{\uI(k)} ) =  
\varphi_d ( \gamma_{\uI(j)} ) \cdot 
\varphi_d ( \gamma_{\uI(1)} \cdots \gamma_{\uI(j-1)} 
            \gamma_{\uI(j+1)} \cdots \gamma_{\uI(k)} ).
\end{equation}  

Consider the permutations
\begin{equation}  \label{eqn:44d}
\sigma_1 := \gamma_{\uI(1)} \cdots \gamma_{\uI(j-1)},
\ \mbox{ and } \ \sigma_2 := \gamma_{\uI(j+1)} \cdots \gamma_{\uI(k)},
\end{equation}  
where in the case $j=1$ (respectively $j=k$) we make the 
convention that $\sigma_1$ (respectively $\sigma_2$) is the 
identity permutation.  The required formula (\ref{eqn:44c}) 
then amounts to 
\[
\varphi_d ( \sigma_1 \gamma_{\uI(j)} \sigma_2 ) = 
\varphi_d ( \gamma_{\uI(j)} ) \cdot \varphi_d ( \sigma_1 \sigma_2 );
\]
Remembering how $\varphi_d$ is defined, we thus see that what we 
have to verify is a relation between two lengths:
\[
|| \sigma_1 \gamma_{\uI(j)} \sigma_2 || 
= 1 + || \sigma_1 \sigma_2 ||.
\]
It is convenient to replace this verification with the 
equivalent one that
\begin{equation}  \label{eqn:44e}
|| \gamma_{\uI(j)} \cdot ( \sigma_2 \sigma_1 ) || 
= 1 + || \sigma_2 \sigma_1 ||,
\end{equation}  
where the equalities $|| \sigma_1 \gamma_{\uI(j)} \sigma_2 || =
|| \gamma_{\uI(j)} \sigma_2 \sigma_1 ) ||$ and 
$|| \sigma_1 \sigma_2 || = || \sigma_2 \sigma_1 ||$
follow from the fact that $|| \cdot ||$ is constant on the 
conjugacy classes of $S_{\infty}$.

In the case when $\sigma_2 \sigma_1$ is the identity permutation,
the equality (\ref{eqn:44e}) holds trivially; so we will assume 
that $\sigma_2 \sigma_1$ is not the identity permutation, and we 
will consider the unique factorization 
\begin{equation}  \label{eqn:44f}
\sigma_2 \sigma_1 = \theta_1 \cdots \theta_p
\end{equation}
where $p \geq 1$ and $\theta_1, \ldots , \theta_p$ are disjoint cycles
of lengths $\ell_1, \ldots , \ell_p \geq 2$.  In particular, this 
gives us the explicit formula
\begin{equation}  \label{eqn:44g}
|| \sigma_2 \sigma_1 || = \sum_{r=1}^p ( \ell_r - 1 ) 
\end{equation}
(following for instance from Equation (\ref{eqn:41c}) in 
Notation \ref{def:41}(3)).

Our hypothesis ``$\uI(j) \neq \uI( \ell )$ for all $\ell \neq j$'' 
implies that $\uI(j) + 1$ is a fixed point of $\sigma_2 \sigma_1$, 
since it is fixed by all the transpositions in the products defining 
$\sigma_1$ and $\sigma_2$.  Hence $\uI(j) + 1$ is not included in any 
of the cycles $\theta_1, \ldots , \theta_p$ from (\ref{eqn:44f}).  
For furher discussion we consider two cases, according to whether 
the number $1$ is or is not included in one of those cycles.

\vspace{6pt}

{\em Case 1.}  $1$ is not a fixed point of $\sigma_2 \sigma_1$, hence it 
is included in one of the cycles $\theta_1, \ldots , \theta_p$.

Since the product of cycles $\theta_1, \ldots , \theta_p$ is a commuting 
one, we may assume without loss of generality (by relabeling the cycles, 
if needed) that $1$ appears in the cycle $\theta_1$.  Then 
\begin{equation}  \label{eqn:44h}
\gamma_{\uI(j)} \sigma_2 \sigma_1 = 
( \gamma_{\uI(j)} \theta_1) \theta_2 \cdots \theta_p,
\end{equation}
where $\gamma_{\uI(j)} \theta_1$ is a cycle of length $1 + \ell_1$ 
(cycling the numbers that were in $\theta_1$ and the number $\uI(j) +1$).  
The right-hand side of (\ref{eqn:44h}) is a disjoint cycle decomposition, 
and the counterpart of Equation (\ref{eqn:44g}) is thus 
\[
|| \gamma_{\uI(j)} \cdot ( \sigma_2 \sigma_1 ) || = ((1 + \ell_1) - 1) +
\sum_{r=2}^p ( \ell_r - 1 ).
\]
By comparing to the right-hand side of (\ref{eqn:44g}) we see that we got
indeed $1 + || \sigma_2 \sigma_1||$, as required.

\vspace{6pt}

{\em Case 2.} $1$ is a fixed point of $\tau$, hence is not included in any
of the cycles $\theta_1, \ldots , \theta_p$.

In this case, $\gamma_{\uI(j)} = (1, \uI(j) + 1)$ commutes with the cycles 
$\theta_1, \ldots , \theta_p$, hence the factorization of 
$\gamma_{\uI(j)} \cdot ( \sigma_2 \sigma_1 )$ into a product of disjoint cycles 
is just $\theta_0 \theta_1 \cdots \theta_p$ with $\theta_0 = \gamma_{\uI(j)}$.
The counterpart of Equation (\ref{eqn:44g}) is thus 
\[
|| \gamma_{\uI(j)} \cdot ( \sigma_2 \sigma_1 ) || = 
\sum_{r=0}^p ( \ell_r - 1 ), \mbox{ with $\ell_0 = 2$,}
\]
and the required equality (\ref{eqn:44e}) follows in this case as well.
\end{proof}

$\ $

\begin{ad-hoc-item}   
{\bf Proof of Proposition \ref{prop:110}.}
The sequence of selfadjoint elements $( \gamma_n )_{n=1}^{\infty}$ in 
the $*$-probability space $( \bC [ S_{\infty} ], \varphi_d )$ fits 
in the framework of Remark \ref{rem:26}, and we can therefore
apply Theorem \ref{thm:25} to the centered sequence 
$( \gamma_n - \frac{1}{d} )_{n=1}^{\infty}$.
Upon doing so, we find precisely the statement of Proposition 
\ref{prop:110}.
\hfill $\square$
\end{ad-hoc-item}

$\ $

\begin{remark}  \label{rem:46}
A direct combinatorial approach to the moments of the 
sums $\gamma_1 + \cdots + \gamma_n$ in the $*$-probability 
space $( \bC [ S_{\infty} ], \varphi_d )$ would involve 
the counting of factorizations of a given permutation 
$\tau \in S_{\infty}$ as product of a specified 
number of star-generators.  More precisely, for every 
$p, n \in \bN$ we have (immediately from the definitions)
that
\begin{equation}   \label{eqn:46a}
\varphi_d \bigl( \, ( \gamma_1 + \cdots + \gamma_n )^p
\, \bigr) = \sum_{\tau \in S_{n+1}} 
c_{n,p} ( \tau ) \, (1/d)^{ || \tau || },
\end{equation}
where $S_{n+1} = \{ \tau \in S_{\infty} \mid 
\tau (k) = k \mbox{ for all } k>n+1 \}$ and where 
for $\tau \in S_{n+1}$ we put
\[
c_{n,p} ( \tau ) :=  \ \vline \ \Bigl\{ 
\uI : \{ 1, \ldots , p \} \to \{ 1, \ldots , n \}
\mid \gamma_{\uI (1)} \cdots \gamma_{\uI (p)} = \tau
\Bigr\} \ \vline \, .
\]
In connection to this, we record here the intriguing
fact that one has precise enumerative formulas for the
related cardinalities
\[
\cciup_{n,p} ( \tau ) := \ \vline \ \Bigl\{
\begin{array}{c}
\uI : \{ 1, \ldots , p \} \to \{ 1, \ldots , n \}, \\
\mbox{ surjective } 
\end{array} \mid 
\gamma_{\uI (1)} \cdots \gamma_{\uI (p)} = \tau
\Bigr\} \ \vline \, , 
\]
counting the so-called ``transitive'' factorizations
of $\tau \in S_{n+1}$ into a product of $p$ 
star-generators.  (See Theorem 1.1 of \cite{GoJa2009}
and the subsequent discussion in Section 1.2.1 of that 
paper, relating to previous results from 
\cite{IrRa2009, Pa1999}.)  We were not able, however, to
use this circle of combinatorial ideas in order to
provide an alternative proof for Theorem \ref{thm:11}
of the present paper.  Such an alternative proof may be 
possible, but let us in any case note that in Equation 
(\ref{eqn:46a}) we need to pursue the case when $p$ 
is fixed and $n \to \infty$, while the factorizations
into star-generators counted by $c_{n,p} ( \tau )$
cannot be transitive for $n > p$.
\end{remark}

$\ $

\section{Connection to GUE matrices 
and proof of Proposition \ref{prop:111} }

\begin{notation}   \label{def:51}
Throughout this section we fix a $d \in \bN$ and we 
continue to use the framework considered in Section 4.  So 
we consider, same as in Proposition \ref{prop:44}, the 
exchangeable sequence $( \gamma_n )_{n=1}^{\infty}$ in the 
$*$-probability space $( \bC [ S_{\infty} ], \varphi_d )$.  
Moreover, same as in Notation \ref{def:19}(1) of the 
Introduction, we let 
$\bu_d : \sqcup_{k=1}^{\infty} \cP (k) \to \bC$
be the function on partitions associated to 
$( \gamma_n )_{n=1}^{\infty}$, and we let 
$\nu^{( \infty , d)} : 
\bC \langle \, \{ X_p \mid p \in \bN \} \, \rangle
\to \bC$ be the linear functional constructed by using 
the function $\bu_d$.

The goal of the section is to point out a connection 
between the exchangeable sequence of (non-centered!) 
star-generators $( \gamma_n )_{n=1}^{\infty}$ and a 
(centered) exchangeable sequence which arises naturally 
in connection to GUE matrices.
\end{notation}

$\ $

\begin{remark-and-notation}   \label{rem:52}
{\em (Sequence of $d \times d$ GUE matrices.) }

\noindent
Let $( \Omega , \cF , P )$ be a probability space, let
$L^{\infty -} ( \Omega, \cF , P )$ be the algebra of complex 
random variables with finite moments of all orders on $\Omega$,  
and let $E : L^{\infty -} ( \Omega , \cF , P) \to \bC$ 
be the expectation functional.  Given a $d \in \bN$ we can then 
consider the $*$-probability space
\begin{equation}  \label{eqn:52a}
\Bigl( \, \cM_d \, \bigl( L^{\infty -} ( \Omega , \cF , P ) \, \bigr),
E \circ \tr_d \, \Bigr),
\end{equation}
where $\cM_d \bigl( L^{\infty -} ( \Omega , \cF , P) \bigr)$ is 
the $*$-algebra of $d \times d$ matrices with entries from 
$L^{\infty -} ( \Omega , \cF , P)$ and where 
$\tr_d : \cM_d \bigl( L^{\infty -} ( \Omega , \cF , P) \bigr)
\to L^{\infty -} ( \Omega , \cF , P)$ is the normalized trace.

Suppose now that in $L^{\infty -} ( \Omega , \cF , P )$ we have
a countable family of independent Gaussian random variables, 
denoted as
\begin{equation}  \label{eqn:57b}
\{ \xi_{i,j}^{(n)} \mid 1 \leq i \leq j \leq d, \ n \in \bN \} 
\cup \{ \eta_{i,j}^{(n)} \mid 1 \leq i < j \leq d, \ n \in \bN \},
\end{equation}
where all the random variables in (\ref{eqn:57b}) are centered
and have variances given by
\[
\mbox{Var} ( \xi_{i,i}^{(n)} ) = \frac{1}{d}, \ \ \forall 
\, 1 \leq i \leq d, \, n \in \bN, \ 
\ \mbox{Var} ( \xi_{i,j}^{(n)} ) 
= \mbox{Var} ( \eta_{i,j}^{(n)} ) = \frac{1}{2d}, \ \ \forall 
\, 1 \leq i < j \leq d, \, n \in \bN .
\]
For every $n \in \bN$, consider the selfadjoint matrix
$M_n \in \cM_d 
\bigl( \, L^{\infty -}( \Omega , \cF , P ) \, \bigr)$
with entries described as follows:

\vspace{6pt}

-- for every $1 \leq i \leq d$, the $(i,i)$-entry of $M_n$ is
$\xi_{i,i}^{(n)}$;

\vspace{6pt}

-- for every $1 \leq i < j \leq d$, the $(i,j)$-entry of $M_n$ is
$\xi_{i,j}^{(n)} + \sqrt{-1} \, \eta_{i,j}^{(n)}$, 

\hspace{0.3cm} and the $(j,i)$-entry of $M_n$ is 
$\xi_{i,j}^{(n)} - \sqrt{-1} \, \eta_{i,j}^{(n)}$.  

\vspace{6pt}

\noindent
These $M_n$'s form what is called a sequence of random  
$d \times d$ {\em GUE matrices} of variance $1$, with 
independent entries.  

For the present paper it will be of relevance to note 
that $( M_n )_{n=1}^{\infty}$ is an exchangeable sequence
of selfadjoint elements in the $*$-probability space
(\ref{eqn:52a}).  One has, moreover, a precise formula 
for computing joint moments of $M_n$'s, which is obtained
by invoking the so-called ``Wick formula'' for moments of 
Gaussian random variables, and is reviewed in Remark 
\ref{rem:54} below.  The formula presented in Remark 
\ref{rem:54} will refer to some specific permutations 
in $S_{\infty}$, so we first take a moment to give names 
to these permutations.
\end{remark-and-notation}

\begin{notation}    \label{def:53}
(1) Let $k = 2h$ be an even positive integer and let 
$\pi = \bigl\{ \, \{ a_1, b_1 \}, \ldots , 
\{ a_h , b_h \} \, \bigr\}$ be a pair-partition in 
$\cP_2 (k)$.  We will denote 
\[
\perm_{\pi} := (a_1, b_1) \cdots (a_h, b_h) \in S_{\infty}
\]
(commuting product of $h$ disjoint transpositions, corresponding to
the $h$ pairs in $\pi$).

\vspace{6pt}

(2) For every $k \in \bN$ we will denote 
\[
\cyctok := (1,2, \ldots , k) \in S_{\infty}
\mbox{ and }
\cyckto := \cyctok^{-1} = (k, \ldots , 2,1) \in S_{\infty}.
\]
In particular, $\mbox{c}_{1 \to 1}$ is the identity permutation
of $\bN$, and $\mbox{c}_{1 \to 2}$ is the first star-generator 
$\gamma_1$. 
\end{notation}

\begin{remark}   \label{rem:54}
{\em (GUE moments via the Wick formula.)}
We now consider again the GUE matrices $(M_n)_{n=1}^{\infty}$ 
from Notation \ref{rem:52}, which we view as a sequence of selfadjoint 
elements in the $*$-probability space from (\ref{eqn:52a}).  
The joint moments of the $M_n$'s can be described as follows: for 
every $k \in \bN$ and $\uI = ( \uI (1), \ldots , \uI (k) ) \in \bN^k$,
one has
\begin{equation}   \label{eqn:54a}
( E \circ \tr_d) ( M_{\uI (1)} \cdots M_{\uI (k)} ) = 
\frac{1}{ d^{(k+2)/2} } 
\, \sum_{ \pi \in \cP_2 (k),  
\ \pi \leq Ker ( \uI )} 
\ d^{\# ( \cyctok  \perm_{\pi} | \{ 1, \ldots , k \})}.
\end{equation}
On the right-hand side of Equation (\ref{eqn:54a}) we have 
used the convention introduced in Notation \ref{def:41}(3): 
``$\# ( \cyctok \perm_{\pi} | \{ 1, \ldots , k \} )$'' stands 
for the number of orbits into which the permutation 
$\cyctok \perm_{\pi} \in S_{\infty}$ breaks the invariant 
finite set $\{ 1, \ldots , k \}$. 

Some concrete examples: when we evaluate 
$( E \circ \tr_d ) (M_1 M_2 M_1 M_2)$, the sum on the right-hand 
side of (\ref{eqn:54a}) has only one term, which is $d^1$, as one
sees upon multiplying 
$(1,2,3,4) \cdot \bigl( (1,3)(2,4) \bigr) = (4,3,2,1) 
\in S_{\infty}$.  But when we evaluate 
$( E \circ \tr_d ) (M_1^4 )$, the sum on the right-hand 
side of (\ref{eqn:54a}) has 3 terms, corresponding to the 3
pairings in $\cP_2(4)$; upon performing the suitable 
multiplications in $S_{\infty}$ we find that these 3 terms add 
up to $2d^3 + d$.  Equation (\ref{eqn:54a}) thus tells us that:
\[
( E \circ \tr_d ) (M_1 M_2 M_1 M_2) = \frac{1}{d^2},
\mbox{ while }
( E \circ \tr_d ) (M_1^4 ) = 2 + \frac{1}{d^2}.
\]

Note that Equation (\ref{eqn:54a}) gives in particular an explicit 
writing for the moments of the common distribution of the $M_n$'s 
in the $*$-probability space (\ref{eqn:52a}).  This distribution is 
called the {\em average empirical eigenvalue distribution} of a GUE 
and was denoted by $\nu_d$ in the statement of Theorem \ref{thm:11} 
in the Introduction.  More precisely, one finds that for every 
$k \in \bN$ one has 
\begin{equation}  \label{eqn:54b}
\int_{\bR} t^k \, d \nu_d (t) = ( E \circ \tr_d )( M_1^k)
= \frac{1}{ d^{(k+2)/2} } \ \sum_{\pi \in \cP_2 (k)} \ 
d^{\# ( \cyctok  \perm_{\pi} | \{ 1, \ldots , k \})},
\end{equation}
with the usual convention that the latter sum is equal to $0$ 
for $k$ odd.  So, for instance, the moments of order 2,4 and 6 of 
$\nu_d$ come out as $1$, $2 + \frac{1}{d^2}$ and respectively 
$5 + \frac{10}{d^2}$, as one can easily see via a direct 
evaluation of the right-hand side of (\ref{eqn:54b}).

For a presentation of how the Wick formula is used in order to derive
Equation (\ref{eqn:54a}), the reader can consult for instance the survey 
paper \cite{Sp2016}.  (See Theorem 2.7 and the calculation preceding it 
in \cite{Sp2016}, where one must also insert the superscript indices 
``$(n)$'' used in our Notation \ref{rem:52} in the description of the 
entries of $M_n$.)
\end{remark}

\begin{remark-and-notation}    \label{rem:55}
Since the summation on the right-hand side of Equation 
(\ref{eqn:54a}) only depends on the partition $\Ker ( \uI )$ 
(rather than depending on $\uI$ itself), it is clear that 
$(M_n)_{n=1}^{\infty}$ is an exchangeable sequence in the 
$*$-probability space from (\ref{eqn:52a}).  We will use the 
notation 
\begin{equation}  \label{eqn:55a}
\bs_d : \sqcup_{k=1}^{\infty} \cP (k) \to \bC
\end{equation}
for the function on partitions associated to this exchangeable 
sequence.  Equation (\ref{eqn:54a}) provides us with an explicit 
formula for $\bs_d$, which reads as follows: for every 
$k \in \bN$ and $\rho \in \cP (k)$, one has 
\begin{equation}  \label{eqn:55b}
\bs_d ( \rho ) =  
\sum_{\pi \in \cP_2 (k), \, \pi \leq \rho} 
(1/d)^{(k+2)/2 - \# (\cyctok \perm_{\pi} | \{ 1, \ldots , k \})}.
\end{equation}

As a special case of Equation (\ref{eqn:55b}), let us observe that
$\bs_d ( \rho ) = 0$ whenever it is not true that all blocks of 
$\rho$ have even cardinality (since in that case the sum on the 
right-hand side of (\ref{eqn:55b}) is empty).  In particular it 
follows that $\bs_d ( \rho ) = 0$ whenever $\rho$ has at least one
block of cardinality $1$, and we thus see that 
$( M_n )_{n=1}^{\infty}$ fulfills all the conditions which are 
needed in the CLT for exchangeable sequences, Theorem \ref{thm:25}.

It comes in handy to also record here that 
for a pair-partition $\pi \in \cP_2 (k)$ one has
\begin{equation}  \label{eqn:55c}     
\bs_d ( \pi ) = (1/d)^{ (k+2)/2 -
         \# ( \cyctok \perm_{\pi} \mid \{ 1, \ldots , k \} ) } .
\end{equation}
Indeed, if $\pi$ is a pair-partition, then the sum on the 
right-hand side of (\ref{eqn:55b}) has only one term,
corresponding to $\pi$ itself.

The main point of the present section is to observe a 
connection between the functions on partitions $\bs_d$ (which
was just defined) and $\bu_d$ (corresponding to the 
star-generators $( \gamma_n )_{n=1}^{\infty}$ of $S_{\infty}$, 
and coming from the Notation \ref{def:19}(2) in the Introduction).  
Of course, $\bs_d$ and $\bu_d$ cannot coincide, since evaluating 
them at the unique partition $\pi_1 \in \cP (1)$ gives 
$\bu_d ( \pi_1 ) = 1/d$ (the common value 
$\varphi_d ( \gamma_n )$ for all $n \in \bN$) 
and $\bs_d ( \pi_1 ) = 0$ (the common value 
$( E \circ \tr_d ) (M_n)$ for all $n \in \bN$). But nevertheless,
it turns out that we have the following proposition.
\end{remark-and-notation} 

$\ $

\begin{proposition}  \label{prop:56}
The functions on partitions $\bs_d$ (from Equation (\ref{eqn:55b}))
and $\bu_d$ (from Notation \ref{def:19}(2)) coincide on 
$\sqcup_{k=1}^{\infty} \cP_2 (k)$.
\end{proposition}

$\ $

If we assume Proposition \ref{prop:56}, then it is easy to write 
down the proof (announced in the title of the section) for 
Proposition \ref{prop:111}.

$\ $

\begin{ad-hoc-item}  
{\bf Proof of Proposition \ref{prop:111},} 
{\em by assuming Proposition \ref{prop:56}.}

\noindent
The linear functional $\nu^{( \infty, d)} : 
\bC \langle \{ X_p \mid p \in \bN \} \rangle \to \bC$ 
was introduced in Notation \ref{def:19}(1) in reference to the 
function on partitions $\bu_d$.  But the definition of
$\nu^{( \infty, d)}$ (following the recipe from Definition
\ref{def:17}(2)) actually uses only the values of $\bu_d$
on pair-partitions.  Due to Proposition \ref{prop:56},
the linear functional $\nu^{( \infty, d)}$ can then be viewed
as coming from the function on partitions $\bs_d$, hence from 
the exchangeable sequence $(M_n)_{n=1}^{\infty}$.  It was 
noticed in Remark \ref{rem:55} that $(M_n)_{n=1}^{\infty}$ 
and $\bs_d$ satisfy the hypotheses of the exchangeable CLT,
Theorem \ref{thm:25}.  This theorem implies that
$\nu^{( \infty, d)}$ is a positive linear functional. 

Moreover, let us fix an $r \in \bN$.  For every even 
$k \in \bN$ and 
$\uI = ( \uI (1), \ldots , \uI (k) ) \in \{ 1, \ldots , r \}^k$, 
the recipe used for writing 
$\nu^{( \infty, d)} ( X_{\uI (1)} \cdots X_{\uI (k)} )$ in terms 
of $\bs_d$ (cf. Definition \ref{def:17}(2)), followed by the 
explicit formula for $\bs_d ( \pi )$ in Equation (\ref{eqn:55c}) 
give us that
\[
\nu^{( \infty, d)} ( X_{\uI (1)} \cdots X_{\uI (k)} )
= \sum_{ \pi \in \cP_2 (k), \ \pi \leq Ker ( \uI ) } \ 
(1/d)^{(k+2)/2 - \# ( \cyctok  \perm_{\pi} 
                      | \{ 1, \ldots , k \} )} .
\]
But the right-hand side of the latter equation can be continued 
with
\[
\begin{array}{ll}
= & ( E \circ \tr_d ) ( M_{\uI (1)} \cdots M_{\uI (k)} ) \ 
   \mbox{ (by Equation (\ref{eqn:54a})) }                       \\
= & \nu^{(r,d)} ( X_{\uI (1)} \cdots X_{\uI (k)} )  \
   \mbox{ (by the definition of $\nurd$, 
           Equation (\ref{eqn:16a})). }
\end{array}
\]
The equality 
$\nu^{( \infty, d)} ( X_{\uI (1)} \cdots X_{\uI (k)} )
= \nu^{(r,d)} ( X_{\uI (1)} \cdots X_{\uI (k)} )$
also holds for any odd $k$ and 
$( \uI (1), \ldots , \uI (k) ) \in \{ 1, \ldots , r \}^k$, 
when both sides of the equality are equal to $0$.  This shows 
that the restriction of $\nu^{( \infty , d)}$ to 
$\bC \langle X_1, \ldots , X_r \rangle$ is equal to 
$\nu^{(r,d)}$, as required. 
\hfill $\square$
\end{ad-hoc-item}

$\ $

We are left with the job of proving Proposition \ref{prop:56}.
We start on this job by observing an explicit formula 
for how $\bu_d$ acts on pair-partitions.  This will be, in 
a certain sense, parallel to the formula recorded for 
$\bs_d$ in Equation (\ref{eqn:55c}), only that we need to use 
a different way of constructing a permutation out of a 
pair-partition.

$\ $

\begin{notation}    \label{def:58}
Let $k = 2h$ be an even positive integer, and let $\pi$ be a 
pair-partition in $\cP_2 (k)$.  We consider the unique way
of writing $\pi$ in the form
\begin{equation}   \label{eqn:58a} 
\left\{  \begin{array}{l}
\pi = \{ V_1, \ldots , V_h \}, \mbox{ with }
V_1 = \{ a_1, b_1 \}, \ldots , V_h = \{ a_h, b_h \}    \\
\mbox{where } a_1 < b_1, \ldots , a_h < b_h \mbox{ and }
a_1 < a_2 < \cdots < a_h.
\end{array}   \right.
\end{equation}
We will use the notation $q_{\pi}$ for the permutation
\begin{equation}   \label{eqn:58b} 
q_{\pi} :=
\gamma_{\uI(1)} \gamma_{\uI(2)} \cdots \gamma_{\uI(2h)} 
\in S_{\infty},
\end{equation}
where the tuple 
$\uI : \{ 1, \ldots , 2h \} \to \{ 1, \ldots , h \}$
is defined by putting
\begin{equation}   \label{eqn:58c} 
\uI( a_1 ) = \uI( b_1 ) = 1, \ldots , 
\uI( a_h ) = \uI( b_h ) = h.
\end{equation}
It is useful to keep in mind that: since the product 
defining $q_{\pi}$ in Equation (\ref{eqn:58b}) only 
uses the star-generators 
$\gamma_1 = (1,2), \ldots , \gamma_h = (1, h+1)$, one 
has $q_{\pi} (m) = m$ for all $m > h+1$.

\vspace{6pt}

\noindent
[A concrete example: say that 
$\pi = \bigl\{ \, \{ 1,5 \}, \, \{ 2,4 \}, \, 
       \{ 3,7 \}, \, \{ 6,8 \} \, \bigr\} \in \cP_2 (8)$.
Then 
\[
q_{\pi} 
= \gamma_1 \gamma_2 \gamma_3 \gamma_2 \gamma_1
  \gamma_4 \gamma_3 \gamma_4                    
= (1,2) (1,3) (1,4) (1,3) (1,2) (1,5) (1,4) (1,5) 
= (3,4,5) \in S_{\infty}.
\]
Note that this is quite different from the permutation 
$p_{\pi}$ associated to $\pi$ in Notation \ref{def:53},
which was simply 
$p_{\pi} = (1,5) (2,4) (3,7) (6,8) \in S_{\infty}$.]
\end{notation}

\begin{lemma}   \label{lemma:59}
For every even $k = 2h \in \bN$ and $\pi \in \cP_2 (k)$
one has:
\begin{equation}   \label{eqn:59a}
\bu_d (\pi) = (1/d)^{ (h+1) - 
          \# ( q_{\pi} \mid \{ 1, \ldots , h+1 \} ) } .
\end{equation}
\end{lemma}

\begin{proof}
The tuple $\uI$ which appeared in the construction of $q_{\pi}$ 
(cf. Equation (\ref{eqn:58c})) has $\Ker (\uI) = \pi$; hence, 
directly from the definition of $\bu_d$, we get that
\[
\bu_d ( \pi )
= \varphi_d \bigl( 
\gamma_{\uI(1)} \cdots \gamma_{\uI(2h)} \bigr) .
\]
This can be followed with
\[
= \varphi_d ( q_{ \pi } )
= (1/d)^{ || q_{ \pi } || }
= (1/d)^{ (h+1) - \# ( q_{\pi} \mid \{ 1, \ldots , h+1 \} ) },
\]
where the third equality invokes the description of 
$|| q_{\pi} ||$ provided by Remark \ref{def:41}(3).
\end{proof}

$\ $

\begin{remark}    \label{rem:515} 

(1) For any even $k = 2h \in \bN$ and $\pi \in \cP_2 (k)$, 
we now have both values $\bs_d ( \pi )$ and $\bu_d ( \pi )$
written as powers of $1/d$, with non-negative integer 
exponents (in Equations (\ref{eqn:55c}) and 
(\ref{eqn:59a}), respectively).
The proof of Proposition \ref{prop:56} is thus reduced to
checking the equality of the said exponents.  That is, 
we are left to verify the following combinatorial formula 
relating the permutations $p_{\pi}$ and $q_{\pi}$:
\begin{equation}   \label{eqn:515a}
(h+1) - \# ( 
\mbox{c}_{1 \to 2h} p_{\pi} | \{ 1, \ldots , 2h \} ) 
= (h+1) - \# ( q_{\pi} | \{ 1, \ldots , h+1 \} ).
\end{equation}

\vspace{6pt}

(2) It is instructive to have a look at the special case 
when the pair-partition $\pi \in \cP_2 (k)$ is non-crossing.
In this case one gets $\bs_d ( \pi ) = \bu_d ( \pi ) =1$,
that is, both sides of Equation (\ref{eqn:515a}) are equal 
to $0$.  This is immediate on the right-hand side, where it 
is easily seen that if $\pi$ is non-crossing then $q_{\pi}$ 
must be the identity permutation in $S_{\infty}$.  (Indeed, 
if $\pi$ is non-crossing then in the canonical writing 
(\ref{eqn:58a}) of $\pi$ there has to be a block 
$V_i = \{ a_i, b_i \}$ with $b_i = 1+ a_i$; so the product
(\ref{eqn:58b}) which defines $q_{\pi}$ has two adjacent 
occurrences of $\gamma_i$ -- remove them and continue by 
induction.)  

The fact that the left-hand side of (\ref{eqn:515a}) is also
equal to $0$ for a non-crossing $\pi$ is less immediate; it
can for instance be explained by using the notion of Kreweras 
complement of $\pi$, and the known formula for the number of 
blocks of the Kreweras complement (see e.g. pages 147-148 in 
Lecture 9 of \cite{NiSp2006}).

\vspace{6pt}

(3) We mention that yet another function on pair-partitions,
related to the same $*$-probability space 
$( \bC [ S_{\infty} ], \varphi_d )$ as considered in the 
present paper, was studied in \cite{BoGu2002} (see Theorem 3.4 of 
\cite{BoGu2002}, where the sequence $\alpha_1, \alpha_2, \ldots$ 
of the theorem has to be specialized to 
$\alpha_1 = \cdots = \alpha_d = 1/d$ and $\alpha_i = 0$ for 
$i > d$).  Denoting the function on pair-partitions studied in 
\cite{BoGu2002} by ``$\bw_d$'', it is also the case that for 
every even $k \in \bN$ and $\pi \in \cP_2 (k)$ one has 
$\bw_d ( \pi ) = (1/d)^m$ for some $m \in \bN \cup \{ 0 \}$, 
with $m = 0$ when $\pi$ is non-crossing.  But experimenting with 
small values of $k$ doesn't suggest a direct connection between 
$\bw_d$ and $\bs_d = \bu_d$ (for instance in the example of 
$\pi \in \cP_2 (8)$ used for illustration at the end of Notation 
\ref{def:58}, one gets $\bs_d ( \pi ) = \bu_d ( \pi ) = (1/d)^2$ 
and $\bw_d ( \pi ) = (1/d)^3$).
\end{remark}

$\ $

In order to complete the proof of Proposition \ref{prop:56},
the ingredient still needed is a good understanding 
of how the permutation $q_{\pi}$ works, for a pairing 
$\pi \in \cP_2 (k)$.  We look more carefully at this, in the 
next remark.

\begin{remark}    \label{rem:510}
Let $\pi \in \cP_2 (k)$ for an even $k = 2h \in \bN$, in 
reference to which we use the notation considered above.  As 
noticed at the end of Notation \ref{def:58}, $q_{\pi}$ can only 
move the numbers from $\{ 1, \ldots , h+1 \}$.  It is instructive 
to examine in detail how $q_{\pi}$ acts on a specified 
$j \in \{ 1, \ldots , h+1 \}$.  We will look at a $j \neq 1$ 
(the case $j=1$ is only slightly different from the others).

Among $\gamma_1, \ldots , \gamma_h$, the only transposition that 
actually moves $j$ is $\gamma_{j-1}$, which appears in the product 
(\ref{eqn:58b}) on positions $a_{j-1}$ and $b_{j-1}$ (here the block 
$V_{j-1} = \{ a_{j-1} , b_{j-1} \}$ is as in the explicit writing of 
$\pi$ from (\ref{eqn:58a})).  When we successively apply the factors
$\gamma_{\uI(2h)}, \gamma_{\uI(2h-1)}, \ldots$  from (\ref{eqn:58b}) 
to $j$, the first time when $j$ is actually moved thus occurs when 
we do
\begin{equation}   \label{eqn:510a}
\gamma_{\uI( b_{j-1} )} (j) = \gamma_{j-1} (j) = 1.
\end{equation}
The value $1$ is then immediately moved by the next factor 
(reading from right to left) in the product, 
$\gamma_{\uI( b_{j-1} - 1)}$, and there are several possible 
cases for how this can go:

\vspace{6pt}

{\em Case 1.} $b_{j-1} - 1$ still belongs to $V_{j-1}$, that is, 
we have $b_{j-1} - 1 = a_{j-1}$.

In this case we get 
$\gamma_{\uI( b_{j-1} - 1)} (1) = \gamma_{\uI( a_{j-1} )} (1) 
= \gamma_{j-1} (1) = j$, and it follows
that $j$ is a fixed point of $q_{\pi}$, since none of the 
factors to the left of $\gamma_{\uI ( a_{j-1} )}$ in the product 
(\ref{eqn:58b}) can move $j$.

\vspace{6pt}

{\em Case 2.} $b_{j-1} - 1$ belongs to $V_{i-1}$ for an $i \neq j$
in $\{ 2, \ldots , h+1 \}$, and $b_{j-1} - 1 = a_{i-1}$.

In this case we get 
$\gamma_{\uI( b_{j-1} - 1)} (1) = \gamma_{\uI( a_{i-1} )} (1) 
= \gamma_{i-1} (1) = i$, and also that $q_{\pi} (j) = i$, since 
none of the factors to the left of $\gamma_{\uI ( a_{i-1} )}$ in 
the product (\ref{eqn:58b}) can move $i$.

\vspace{6pt}

{\em Case 3.} $b_{j-1} - 1$ belongs to $V_{i-1}$ for an $i \neq j$
in $\{ 2, \ldots , h+1 \}$, and $b_{j-1} - 1 = b_{i-1}$.

In this case we still get (same as in Case 2) 
$\gamma_{\uI( b_{j-1} - 1)} (1) = i$, but we cannot yet decide what 
is $q_{\pi} (j)$, since the number $i$ will be moved by the later 
factor $\gamma_{\uI( a_{i-1} )}$ of the product (\ref{eqn:58b}).
More precisely, Case 3 can be divided into subcases according to the 
status of $a_{i-1}$, as follows.

\vspace{6pt}

{\em Case 3-1.} $a_{i-1} = 1$ (which, according to the convention 
from (\ref{eqn:58a}), means that $i=2$).  

Then $\gamma_{\uI( a_{i-1} )} (i) = \gamma_{i-1} (i) = 1$, and we 
conclude that $q_{\pi} (j) = 1$.

\vspace{6pt}

{\em Case 3-2.} $a_{i-1} \neq 1$, and $a_{i-1} -1 = a_{\ell -1}$ for
\footnote{The value of $\ell$ appearing here is sure to be such that 
$\ell \neq i$, but it is not ruled out that we have $\ell = j$.  In 
the latter case, the outcome of Case 3-2 is that $j$ is a fixed point 
for $q_{\pi}$. }
an $\ell \in \{ 2, \ldots , h+1 \}$.

In this subcase we find that
$\gamma_{\uI( a_{i-1})} (i) = \gamma_{i-1} (i) = 1$, followed by 
$\gamma_{\uI( a_{\ell -1})} (1) = \gamma_{\ell -1} (1) = \ell$.  At 
this point we can conclude that $q_{\pi} (j) = \ell$, because $\ell$ 
is no longer moved by the remaining factors (to the left of 
$\gamma_{\uI( a_{\ell -1})}$) to be considered in the product 
(\ref{eqn:58c}).

\vspace{6pt}

{\em Case 3-3.} $a_{i-1} \neq 1$, and $a_{i-1} -1 = b_{\ell -1}$ for
an $\ell \in \{ 2, \ldots , h+1 \}$.

In this subcase we find, same as in Case 3-2, that
$\gamma_{\uI( a_{i-1} )} (i) = 1$, which is now followed by 
$\gamma_{\uI( b_{ \ell -1} )} (1) = \ell$.  But unlike in Case 3-2, 
in order to continue the discussion towards the determination of 
$q_{\pi} (j)$, we need to make a further subdivision into 
subcases.  Indeed, what we must do is look at the number 
$a_{\ell -1} < b_{\ell -1}$ and break again into three subcases (which 
could be numbered as Cases 3-3-1, 3-3-2 and 3-3-3) according to whether 
$a_{\ell - 1} = 1$, or $a_{\ell - 1} -1 = a_{m-1}$ for some 
$m \in \{ 2, \ldots , h+1 \}$, or $a_{\ell - 1} -1 = b_{m-1}$ for some 
$m \in \{ 2, \ldots , h+1 \}$.

The conclusion of all this discussion will be recorded in 
Lemma \ref{lemma:512} below.  In order to state the lemma, we first
clarify the convention for how to restrict permutations to subsets 
that aren't necessarily invariant.
\end{remark}

\begin{notation}    \label{def:511}
Let $\tau$ be a permutation in $S_{\infty}$, and let $A$ be a finite 
non-empty subset of $\bN$.  We do not assume that $A$ is invariant for 
$\tau$, but let us note that there still exists a natural bijection 
$\theta : A \to A$ which one could call ``permutation of $A$ induced by   
$\tau$'', and is described as follows. 

Let $a$ be a number in $A$.  We look at the sequence of values 
$v_1, v_2, \ldots , v_k , \ldots$ in $\bN$ obtained by putting
\begin{equation}   \label{eqn:511a}
v_1 = \tau (a), v_2 = \tau (v_1), \ldots , v_k = \tau (v_{k-1}), \ldots
\end{equation}
and we define $\theta (a) := v_{k_o} \in A$ where
$k_o := \min \{ k \in \bN \mid v_k \in A \}$.  In words: the $v_k$'s
in (\ref{eqn:511a}) follow the orbit of $\tau$ which contains $a$, 
and $\theta (a)$ is the first re-entry in $A$ which is encountered 
along that orbit.
\end{notation}

We leave to the reader the (fairly straightforward but nevertheless 
tedious) job to verify that, upon re-reading and suitably expanding the 
multi-case discussion about ``how to find out what is $q_{\pi} (j)$'' 
from Remark \ref{rem:510}, one arrives to the following formal statement.

\begin{lemma}   \label{lemma:512}
Let $k = 2h$ be an even positive integer and let 
$\pi = \{ V_1, \ldots , V_h \}$ be a pair-partition in $\cP_2 (2h)$,
where we write explicitly
$V_1 = \{ a_1, b_1 \}, \ldots , V_h = \{ a_h, b_h \}$,
with $a_1 < b_1, \ldots , a_h < b_h$ and with
$1 = a_1 < \cdots < a_h$. We also put $b_0 := 2h+1$.  We consider two 
permutations $\theta_1, \theta_2$ of sets of cardinality $h+1$, 
as follows.

\vspace{6pt}

$\bullet$  On the one hand, let $q_{\pi} \in S_{\infty}$ be 
defined as in Equation (\ref{eqn:58b}) of Notation \ref{def:58}, and 
let $\theta_1$ be the restriction of $q_{\pi}$ to its invariant 
set $\{ 1, \ldots , h+1 \}$.

\vspace{6pt}

$\bullet$
On the other hand let us consider the permutations
$\perm_{\pi}$ and $\mbox{c}_{2h+1 \to 1}$ defined 
in Notation \ref{def:53}, and let $\theta_2$ be the permutation 
induced (in the sense of Notation \ref{def:511}) by the product 
$\perm_{\pi} \cdot \mbox{c}_{2h+1 \to 1} \in S_{\infty}$
on the finite set $\{ b_0, b_1, \ldots , b_h \} \subseteq \bN$.  

\vspace{6pt}

Then $\theta_1$ and $\theta_2$ are conjugated by the bijection 
$\{ 1,2, \ldots , h+1 \} \to \{ b_0, b_1, \ldots , b_h \}$
which maps $i \mapsto b_{i-1}$ for every $1 \leq i \leq h+1$.
\hfill  $\square$
\end{lemma}

In the framework of the preceding lemma, it will be useful to also 
have on record the following observation.

\begin{lemma}   \label{lemma:513}
Consider the same notations as in Lemma \ref{lemma:512}. 
Suppose $R$ is an orbit of the permutation 
$\perm_{\pi} \cdot \mbox{c}_{2h+1 \to 1}$,
such that $R \subseteq \{ 1,2, \ldots , 2h + 1 \}$.  Then 
$R \cap \{ b_0, b_1, \ldots , b_h \} \neq \emptyset$.
\end{lemma}

\begin{proof}
Assume for contradiction that 
$R \cap \{ b_0, b_1, \ldots , b_h \} = \emptyset$.
This implies that 
\[
R \subseteq \{ 1,2, \ldots , 2h+1 \} \setminus 
\{ b_0, b_1, \ldots , b_h \} = \{ a_1, \ldots , a_h \}.
\]

Let us observe that, due to the convention of how 
$a_1, \ldots , a_h$ are chosen, we have the implication
\[
\left(  \begin{array}{c}
2 \leq i \leq h  \\
( \perm_{\pi} \cdot \mbox{c}_{2h+1 \to 1} ) (a_i) = a_j
\end{array}  \right)  \ \Rightarrow \ a_j < a_i \ \Rightarrow \ j < i.
\]
By starting with an $a_i \in R$ and by iterating the above 
observation, we must eventually get that $a_1 \in R$.  
But $a_1 = 1$, and
$\bigl( \, \perm_{\pi} \cdot \mbox{c}_{2h+1 \to 1} \, \bigr) (1)
= \perm_{\pi} (2h+1) = 2h+1 = b_0$.  It follows that $b_0 \in R$,
contradiction.
\end{proof}

$\ $

\begin{ad-hoc-item}
{\bf Proof of Proposition \ref{prop:56}.}
We fix an even positive integer
$k = 2h$ and a pair-partition $\pi \in \cP_2 (2h)$.
We consider the canonical writing
$\pi = \{ \, \{ a_1, b_1 \}, \ldots , \{ a_h, b_h \} \, \}$
described in Notation \ref{def:58}.  As observed in 
Remark \ref{rem:515}(1), we are left 
to verify a combinatorial identity, coming to:
\begin{equation}   \label{eqn:514a}
\# \Bigl( q_{\pi} \mid \{ 1, \ldots , h+1 \} \Bigr)
= \# \Bigl( \perm_{\pi} \cdot \mbox{c}_{1 \to 2h} \mid 
\{ 1, \ldots , 2h \} \Bigr).
\end{equation}
We will do this verification by checking that both sides of 
Equation (\ref{eqn:514a}) are equal to
\begin{equation}   \label{eqn:514b}
\# \bigl( \perm_{\pi} \cdot \mbox{c}_{2h+1 \to 1} \mid
\{ 1, \ldots , 2h+1 \} \bigr).
\end{equation}

\vspace{6pt}

{\em Verification that the left-hand side of (\ref{eqn:514a}) 
is equal to (\ref{eqn:514b}). }
Same as in Lemma \ref{lemma:512}, besides the numbers 
$b_1, \ldots , b_h$ considered in the canonical writing of $\pi$ 
we also put $b_0 := 2h+1$.  Lemma \ref{lemma:512} says that the 
permutation induced by $\perm_{\pi} \cdot \mbox{c}_{2h+1 \to 1}$ 
on $\{ b_0, b_1, \ldots , b_h \}$ is precisely what one obtains 
by starting from $q_{\pi} \mid \{ 1, \ldots , h+1 \}$
and by doing the identification $i \leftrightarrow b_{i-1}$, 
$1 \leq i \leq h+1$.  This implies that 
$\# \bigl( q_{\pi} \mid \{ 1, \ldots ,h+1 \} \bigr)$ 
is equal to the number of orbits of 
$\perm_{\pi} \cdot \mbox{c}_{2h+1 \to 1}$ which intersect 
$\{ b_0, b_1, \ldots , b_h \}$.  However, Lemma \ref{lemma:513} 
assures us that these are all the orbits of 
$\perm_{\pi} \cdot \mbox{c}_{2h+1 \to 1} \mid \{ 1, \ldots , 2h+1 \}$
(and this concludes the required verification).

\vspace{6pt}

{\em Verification that the right-hand side of (\ref{eqn:514a}) 
is equal to (\ref{eqn:514b}). }
Here we first observe the equality
\begin{equation}   \label{eqn:514c}
\# \bigl( \perm_{\pi} \cdot \mbox{c}_{2h+1 \to 1} \mid 
\{ 1, \ldots , 2h+1 \} \bigr) = 
\# \bigl( \perm_{\pi} \mbox{c}_{1 \to 2h+1} \mid 
\{ 1, \ldots , 2h+1 \} \bigr),
\end{equation}
which is an immediate consequence of the fact that $\perm_{\pi}$ 
is its own inverse.  So it suffices to verify the equality between 
the right-hand side of (\ref{eqn:514a}) and the right-hand side of 
(\ref{eqn:514c}).  The latter equality comes out of the following 
observation: the action of the permutation 
$\perm_{\pi} \cdot \mbox{c}_{1 \to 2h+1}$ is such that
\[
2h \mapsto 2h+1 \mapsto b_1;
\]
this shows that 
$\perm_{\pi} \cdot \mbox{c}_{1 \to 2h+1} \mid \{ 1, \ldots , 2h+1 \}$ 
can be obtained by starting from 

\noindent
$\perm_{\pi} \cdot \mbox{c}_{1 \to 2h} \mid \{ 1, \ldots , 2h \}$ 
and by inserting the number $2h+1$ in the cycle which contains $2h$ 
(from ``$( \ldots , 2h, b_1, \ldots )$'', that cycle becomes 
``$( \ldots , 2h, 2h+1, b_1, \ldots )$'').  It follows that 
$\perm_{\pi} \cdot \mbox{c}_{1 \to 2h+1} \mid \{ 1, \ldots , 2h+1 \}$ 
and $\perm_{\pi} \cdot \mbox{c}_{1 \to 2h} \mid \{ 1, \ldots , 2h \}$ 
have the same number of orbits, as required.
\hfill $\square$
\end{ad-hoc-item}

$\ $

$\ $

{\bf Acknowledgements.}  
We are grateful to the anonymous referees for the 
careful reading and for thought-provoking comments 
which lead to significant improvements in the 
first version of the paper.
A.N. would also like to thank Alexandru Gatea for 
useful discussions at an early stage of this project.

$\ $

\end{document}